\documentclass[11pt]{article}
\usepackage{amsfonts}
\usepackage{amsmath,amssymb}
\usepackage{amsthm}
\usepackage[mathscr]{euscript}
\usepackage{mathrsfs}
\usepackage{indentfirst}
\usepackage{color}\usepackage{enumitem}

\newtheorem{theorem}{Theorem}[section]
\newtheorem{lemma}[theorem]{Lemma}

\newtheorem{remark}[theorem]{Remark}

\numberwithin{equation}{section}

\allowdisplaybreaks

\newcommand{\be}{\begin{equation}}
\newcommand{\ee}{\end{equation}}
\newcommand\bes{\begin{eqnarray}} \newcommand\ees{\end{eqnarray}}
\newcommand{\bess}{\begin{eqnarray*}}
\newcommand{\eess}{\end{eqnarray*}}
\newcommand{\bbbb}{\left\{\begin{aligned}}
\newcommand{\nnnn}{\end{aligned}\right.}
\newcommand{\bea}{\begin{align*}}
\newcommand{\eea}{\end{align*}}

\newcommand\ep{\varepsilon}

\newcommand\kk{\left}
\newcommand\rr{\right}
\newcommand\dd{\displaystyle}

\newcommand\dx{{\rm d}x}
\newcommand\dy{{\rm d}y}

\newcommand\yy{\infty}

\newcommand\sk{\smallskip}

\newcommand\rrr{\color{red}}
\begin{document}\thispagestyle{empty}
\setlength{\baselineskip}{16pt}

\begin{center}
 {\LARGE\bf Free boundary problems of a mutualist model with nonlocal diffusions\footnote{This work was supported by NSFC Grants
11771110, 11971128}}\\[4mm]
  {\Large Lei Li, \ \  Mingxin Wang\footnote{Corresponding author. {\sl E-mail}: mxwang@hit.edu.cn}}\\[1.5mm]
{School of Mathematics, Harbin Institute of Technology, Harbin 150001, PR China}
\end{center}

\date{\today}

\begin{abstract}
A mutualist model with nonlocal diffusions and a free boundary is first considered. We prove that this problem has a unique solution defined for $t\ge0$, and its dynamics are governed by a spreading-vanishing dichotomy. Some criteria for spreading and vanishing are also given. Of particular importance is that we find that the solution of this problem has quite rich longtime behaviors, which vary with the conditions satisfied by kernel functions and are much different from those of the counterpart with local diffusion and free boundary. At last, we extend these results to the model with nonlocal diffusions and double free boundaries.



\end{abstract}

\section{Introduction}
We investigate the following mutualist model with nonlocal diffusions and a free boundary
\bes\left\{\begin{aligned}\label{1.1}
&u_t=d_1\int_{0}^{\yy}J_1(x-y)u(t,y)\dy-d_1j_1(x)u+f_1(u,v), & &t>0,~x\in\overline{\mathbb{R}}^+,\\
&v_t=d_2\int_{0}^{h(t)}J_2(x-y)v(t,y)\dy-d_2j_2(x)v+f_2(u,v), & &t>0,~x\in[0,h(t)),\\
&v(t,x)=0,& &t>0, ~ x\in[h(t),\yy),\\
&h'(t)=\mu\int_{0}^{h(t)}\int_{h(t)}^{\infty}
J_2(x-y)v(t,x)\dy\dx,& &t>0,\\
&h(0)=h_0>0;\;\;u(0,x)=u_0(x),\;\; x\in\overline{\mathbb{R}}^+;\;\;v(0,x)=v_0(x),&&x\in[0,h_0]
 \end{aligned}\right.
 \ees
with
 \[f_1(u,v)=r_1u\kk(a-u-\frac{u}{1+bv}\rr),\;\;\;f_2(u,v)=r_2v\kk(1-v-\frac{v}{1+cu}\rr),\]
where all parameters are positive, $j_i(x)=\int_{0}^{\yy}J_i(x-y)\dy$ for $i=1,2$, $0<u_0\in C(\overline{\mathbb{R}}^+)\cap L^{\yy}(\overline{\mathbb{R}}^+)$, $v_0\in C([0,h_0])$ and $v_0(h_0)=0<v_0(x)$ in $[0,h_0)$. The kernel function $J_i$ for $i=1,2$ satisfy
 \begin{enumerate}[leftmargin=4em]
\item[{\bf(J)}]$J\in C(\mathbb{R})\cap L^{\yy}(\mathbb{R})$, $J\ge 0$, $J(0)>0,~\dd\int_{\mathbb{R}}J(x)\dx=1$, \ $J$\; is even.
 \end{enumerate}

Such mutualist model with random diffusions and the Stefan type boundary condition has been proposed and studied in \cite{LL,ZW}. The existence and uniqueness of global solution, longtime behaviors as well as criteria for spreading and vanishing have been established by a series of classical analysis. Like this, the application of the Stefan type problem in ecology has recently attracted much attention from many researchers over the past decades since the pioneering work of \cite{DL2010}. For example, one can refer to \cite{DHW,DG} for problems in high dimensional space, \cite{Wjfa16,DDL1,DDL2} for the model in time or space periodic setting, and \cite{GW,WZjdde14} for two species model.

Our interest in this paper is to study the dynamics of nonlocal diffusion problem \eqref{1.1}.  As is known to us, one of the advantages of nonlocal diffusion over the random diffusion is that it can be used to describe the movements of organisms between not only adjacent spatial locations but nonadjacent locations, which is expected to bring about some diverse dynamics; see, e.g. \cite{BCV,KLS}. To check the difference between dynamics of local and nonlocal diffusion models with free boundary, the authors of \cite{CDLL} and \cite{CQW} proposed the following nonlocal diffusion model with free boundaries
\bes\label{1.2}\left\{\begin{aligned}
&u_t=d\int_{g(t)}^{h(t)}J(x-y)u(t,y)\dy-du+f(u), & &t>0,~g(t)<x<h(t),\\
&u(t,x)=0,& &t>0, ~ x\notin(g(t),h(t)),\\
&h'(t)=\mu\int_{g(t)}^{h(t)}\int_{h(t)}^{\infty}
J(x-y)u(t,x)\dy\dx,& &t>0,\\
&g'(t)=-\mu\int_{g(t)}^{h(t)}\int_{-\infty}^{g(t)}
J(x-y)u(t,x)\dy\dx,& &t>0,\\
&h(0)=-g(0)=h_0>0,\;\; u(0,x)=u_0(x),& &|x|\le h_0,
 \end{aligned}\right.
 \ees
where $J$ satisfies {\bf(J)}. The reaction term $f(u)$ in \cite{CDLL} is of the Fisher-KPP type, while that in \cite{CQW} is identical to zero, which results in very different dynamics between these two models.

Since the above two works, lots of related researches have recently emerged. Du et.al \cite{DLZ} considered the spreading speed of free boundaries in \cite{CDLL} when spreading happens. They found an important condition on $J$, namely,
 \begin{enumerate}[leftmargin=4em]
\item[{\bf(J1)}]$\dd\int_{0}^{\yy}xJ(x)\dx<\yy$,
 \end{enumerate}
 so that there is a unique finite spreading speed for the free boundaries if and only if {\bf(J1)} holds for $J$. More precisely, they obtained that when $\lim_{t\to\yy}-g(t)=\lim_{t\to\yy}h(t)=\yy$,
  \begin{align*}
\lim_{t\to\yy}\frac{-g(t)}{t}=\lim_{t\to\yy}\frac{h(t)}{t}=\left\{\begin{aligned}
&C_0& &{\rm if~{\bf(J1)}~holds~for~}J,\\
&\yy& &{\rm if~{\bf(J1)}~is~not~satisfied},
\end{aligned}\right.
  \end{align*}
  where $C_0$ is uniquely given by the semi-wave problem
  \bes\label{1.3}\left\{\begin{array}{lll}
 d\dd\int_{-\yy}^{0}J(x-y)\phi(y)\dy-d\phi+c\phi'+f(\phi)=0, \quad -\yy<x<0,\\[1mm]
\phi(-\yy)=U^*,\ \ \phi(0)=0, \ \ c=\mu\dd\int_{-\yy}^{0}\int_{0}^{\yy}J(x-y)\phi(x)\dy\dx,
 \end{array}\right.
 \ees
with $U^*$ denoting the unique positive zero of $f(u)$.
 Moreover, Du and Ni \cite{DN21} not only extended the above results to monostable cooperative system, but also gave some accurate estimates on free boundary. In particular, if $J\approx|x|^{-\gamma}$ for $|x|\gg1$ and some $\gamma\in(1,2]$, that is, $C_1|x|^{-\gamma}\le J(x)\le C_2|x|^{-\gamma}$ for some $C_1,C_2>0$ and $|x|\gg1$, then $-g(t),h(t)\approx t^{1/(\gamma-1)}$ for $\gamma\in(1,2)$ and $-g(t),h(t)\approx t\ln t$ for $\gamma=2$. The high dimension and radial symmetry version was also considered in \cite{DN213}, where new difficulties arising from the kernel function $J(|x|)$ have been solved by using some delicate techniques. Among other things, there are other works about nonlocal diffusion problem with free boundaries; please see e.g. \cite{DWZ,DN20,DN200,WW2,ZZLD,LSW} and references therein.

Inspired by the deduction of model \eqref{1.2}, very recently Li and Wang \cite{LW21} put forward the following nonlocal diffusion model with a free boundary
  \bes\label{1.4}
\left\{\begin{aligned}
&u_t=d\dd\int_{0}^{h(t)}J(x-y)u(t,y)\dy-d\dd\left(\int_{0}^{\yy}J(x-y)\dy\right) u+f(u), && t>0,~0\le x<h(t),\\
&u(t,h(t))=0, && t>0,\\
&h'(t)=\mu\dd\int_{0}^{h(t)}\int_{h(t)}^{\infty}
J(x-y)u(t,x)\dy\dx, && t>0,\\
&h(0)=h_0,\;\; u(0,x)=u_0(x), &&x\in[0,h_0].
\end{aligned}\right.
 \ees
They proved that this problem has the similar dynamics with those in \cite{CDLL,DLZ}. Later on, Li and Wang \cite{LW211} obtained the following sharp estimates on solution of \eqref{1.4} when spreading happens
 \bess\left\{\begin{aligned}
&\lim_{t\to\yy}\max_{x\in[0,\,ct]}\big|u(t,x)-u^*\big|=0 ~ {\rm for ~ any ~ } c\in(0,c_0), ~ {\rm  ~ if ~ J ~ satisfies ~ {\bf(J1)}},\\
&\lim_{t\to\yy}\max_{x\in[0,\,ct]}\big|u(t,x)-u^*\big|=0 ~ {\rm for ~ any ~ } c>0, ~ {\rm  ~ if ~ J ~ violates ~ {\bf(J1)}},\\
&\lim_{t\to\yy}\max_{x\in[0,\,s(t)]}|u(t,x)- u^*|=0 ~ {\rm for ~ any } ~ 0\le s(t)=t^{\frac{1}{\gamma-1}}o(1) ~ ~ {\rm  ~ if ~}J\approx|x|^{-\gamma} ~ {\rm for ~ } \gamma\in(1,2),\\
&\lim_{t\to\yy}\max_{x\in[0,\,s(t)]}|u(t,x)- u^*|=0 ~ {\rm for ~ any } ~ 0\le s(t)=(t\ln t)o(1) ~ ~ {\rm  ~ if ~}J\approx|x|^{-2}.
\end{aligned}\right.\eess

Motivated by the above works, our goal here is to check whether the above results can be extended to problem \eqref{1.1} and its variation with double free boundaries. This paper is arranged as follows. Section 2 is devoted to longtime behaviors of \eqref{1.1}. In Section 3, criteria for spreading and vanishing are given. Section 4 deals with dynamics of the variation of \eqref{1.1} where $u$ is distributed through the whole space $\mathbb{R}$ and $v$ diffuses with spreading fronts $\{x=g(t)\}$ and $\{x=h(t)\}$.
\section{Longtime behaviors of \eqref{1.1}}
\setcounter{equation}{0} {\setlength\arraycolsep{2pt}
First of all, the following well-posedness of \eqref{1.1} can be obtained by the similar arguments in proof of \cite[Theorem 2.1]{CDLL} and \cite[Theorem 2.2]{CLWZ}, and thus its proof is omitted here.
\begin{theorem}\label{t2.1} Problem \eqref{1.1} has a unique global solution $(u,v,h)$, and $u,u_t\in C(\overline{\mathbb{R}}^+\times\overline{\mathbb{R}}^+)$, $v,v_t\in C(\overline{\mathbb{R}}^+\times [0,h(t)])$ and $h\in C^1([0,\yy))$. Moreover, the following estimates hold:
  \[0\le u\le K_1:=\max\{\|u_0\|_{L^{\yy}(\overline{\mathbb{R}}^+)}, \ a\}, \ \ \ 0\le v\le K_2:=\max\{\|v_0\|_{C([0,h_0])}, \ 1\}.\]
 \end{theorem}

By the equation of $h'$, $h(t)$ is increasing in $t>0$. Thus $h_{\yy}:=\lim_{t\to\yy}h(t)\in(0,\yy]$.  The case $h_{\yy}<\yy$ is called by vanishing case, while the other case $h_{\yy}=\yy$ is described as spreading case. The upcoming study of longtime behaviors of solution to \eqref{1.1} will be divided into these two cases.

 \subsection{Vanishing case: $h_{\yy}<\yy$}

Let $a_0$ and $l$ be positive constants. It is well-known that eigenvalue problem
 \bess
(\mathcal{L}_{(0,\,l)}+a_0) \phi:=d_2\int_{0}^{l}J_2(x-y)\phi(y)\dy-d_2j_2(x)\phi+a_0\phi=\lambda\phi ~ ~ \mbox{in}\;\;[0,l]
\eess
has a unique principal eigenvalue, denoted by $\lambda_p(\mathcal{L}_{(0,\,l)}+a_0)$, whose corresponding eigenfunction is positive in $[0,l]$. Please see \cite[Lemma 2.3]{LW21} for some properties of $\lambda_p(\mathcal{L}_{(0,\,l)}+a_0)$.

 Moreover, it follows from \cite[Proposition 1.3]{DLZ} that if $J_1$ satisfies the condition
 \begin{enumerate}[leftmargin=4em]
\item[{\bf(J2)}]$\dd\int_{0}^{\yy}J(x)e^{\lambda x}\dx<\yy$ for some $\lambda>0$,
 \end{enumerate}
 then one can find a unique $c_*>0$, usually called by the minimal wave speed, so that problem
 \bes\label{2.0}\left\{\begin{array}{lll}
 d_1\dd\int_{-\yy}^{\yy}J_1(x-y)\phi(y)\dy-d_1\phi+c\phi'+r_1\phi(a-2\phi)=0, \quad x\in\mathbb{R},\\[1mm]
\phi(-\yy)={a}/{2},\ \ \phi(\yy)=0, \ \ \phi'(x)\le0
 \end{array}\right.
 \ees
has a unique solution $\phi_c$ if and only if $c\ge c_*$. Below is the main result of this subsection.
\begin{theorem}\label{t2.2} If $h_{\yy}<\yy$, then $\lambda_p(\mathcal{L}_{(0,\,h_{\yy})}+r_2)\le 0$, $\lim_{t\to\yy}\|v(t,\cdot)\|_{C([0,h(t)])}=0$ and
 \bes\label{2.1}\left\{\begin{aligned}
&\lim_{t\to\yy}\max_{x\in[0,\,ct]}\kk|u(t,x)-a/2\rr|=0 ~ {\rm for ~ any ~ } c\in(0,c_*), ~ {\rm  ~ if ~ J_1 ~ satisfies ~ {\bf(J2)}},\\
&\lim_{t\to\yy}\max_{x\in[0,\,ct]}\kk|u(t,x)-a/2\rr|=0 ~ {\rm for ~ any ~ } c>0, ~ {\rm  ~ if ~ J_1 ~ violates ~ {\bf(J2)}},\\
&\lim_{t\to\yy}\max_{x\in[0,\,s(t)]}\kk|u(t,x)-a/2\rr|=0 ~ {\rm for ~ any } ~ 0\le s(t)=t^{\frac{1}{\gamma-1}}o(1) ~ ~ {\rm  ~ if ~}J_1\approx|x|^{-\gamma} ~ {\rm for ~ } \gamma\in(1,2),\\
&\lim_{t\to\yy}\max_{x\in[0,\,s(t)]}\kk|u(t,x)-a/2\rr|=0 ~ {\rm for ~ any } ~ 0\le s(t)=(t\ln t)o(1) ~ ~ {\rm  ~ if ~}J_1\approx|x|^{-2}.
\end{aligned}\right.\ees
 \end{theorem}
 \begin{proof}{\bf Step 1}: {\it The proof of $\lambda_p(\mathcal{L}_{(0,\,h_{\yy})}+r_2)\le 0$}. Assume indirectly that $\lambda_p(\mathcal{L}_{(0,\,h_{\yy})}+r_2)> 0$. It follows from {\bf(J)} that there are $\ep_0$ and $\sigma_0>0$ such that $J(x)>\sigma_0$ for $|x|\le \ep_0$. Then, by continuity, there exist $\ep\in(0,\ep_0/2)$ and $T>0$ such that $\lambda_p(\mathcal{L}_{(0,\,h_{\yy}-\ep)}+r_2)>0$ and $h(t)>h_{\yy}-\ep$ for $t\ge T$. Let $w$ be a solution of
\bess\left\{\begin{aligned}
&w_t= d_2\int_{0}^{h_{\yy}-\ep}J_2(x-y)w(t,y)\dy-dj_2(x)w+r_2w(1-2w), && t>T,\ 0\le x\le h_{\yy}-\ep,\\
&w(T,x)=v(T,x)  && 0\le x\le h_{\yy}-\ep.
 \end{aligned}\right.\eess
From a simple comparison argument, $v\ge w$ for $t\ge T$ and $x\in [0,h_{\yy}-\ep]$.
In view of $\lambda_p(\mathcal{L}_{(0,\,h_{\yy}-\ep)}+r_2)>0$ and \cite[Lemma 2.6]{LW21}, we have that $w$ converges to a positive steady state $W(x)$ uniformly in $[0,h_{\yy}-\ep]$ as $t\to\yy$. Thus there exists $T_1>T$ such that $v(t,x)\ge W(x)/2$ for $t\ge T_1$ and $x\in [0,h_{\yy}-\ep]$. Furthermore, we see that for $t>T_1$,
  \bess
h'(t)&=&\mu\dd\int_{0}^{h(t)}\int_{h(t)}^{\infty}
J_2(x-y)v(t,x)\dy\dx\ge\mu\dd\int_{h_{\yy}-\frac{\ep_0}{2}}^{h_{\yy}-\ep}\int_{h_{\yy}}^{h_{\yy}+\frac{\ep_0}{2}}
J_2(x-y)v(t,x)\dy\dx\\
&\ge&\mu\dd\int_{h_{\yy}-\frac{\ep_0}{2}}^{h_{\yy}-\ep}\int_{h_{\yy}}^{h_{\yy}+\frac{\ep_0}{2}}
\sigma_0 \frac{W(x)}{2}\dy\dx>0,
\eess
which contradicts to $h_{\yy}<\yy$. Consequently, $\lambda_p(\mathcal{L}_{(0,\,h_{\yy})}+r_2)\le0$.

{\bf Step 2}: {\it The proof of $\lim_{t\to\yy}\|v(t,\cdot)\|_{C([0,h(t)])}=0$}. Consider the following problem
\bess\left\{\begin{aligned}
&V_t= d_2\int_{0}^{h_{\yy}}J_2(x-y)V(t,y)\dy-dj_2(x)V+r_2V(1-V), && t>0,\ 0\le x\le h_{\yy},\\
&V(0,x)=K_2  && 0\le x\le h_{\yy}.
 \end{aligned}\right.\eess
The comparison principle indicates that $v\le V$ for $t\ge0$ and $0\le x\le h(t)$. By $\lambda_p(\mathcal{L}_{(0,\,h_{\yy})}+r_2)\le0$, we see from \cite[Lemma 2.6]{LW21} that $\lim_{t\to\yy}V(t,x)=0$ in $C([0,h_{\yy}])$. Thus, we complete this step.

 {\bf Step 3}: {\it The proof of \eqref{2.1}}. Clearly, for any small $\ep>0$ there is $T>0$ such that $v\le \ep$ for $t\ge T$ and $x\ge 0$. Thus $u$ satisfies
 \bess\left\{\begin{aligned}
&u_t\le d_1\int_{0}^{\yy}J_1(x-y)u(t,y)\dy-d_1j_1(x)u+f_1(u,\ep), && t>T,\ x\in \overline{\mathbb{R}}^+.\\
&u(T,x)\le K_1, && x\in \overline{\mathbb{R}}^+.
 \end{aligned}\right.\eess
A simple comparison principle arrives at $\limsup_{t\to\yy}u(t,x)\le a(1+b\ep)/(2+b\ep)$ uniformly in $\overline{\mathbb{R}}^+$. Thanks to the arbitrariness of $\ep$, we have
 \bes\label{2.2}\limsup_{t\to\yy}u(t,x)\le{a}/{2}\;\; ~ {\rm uniformly ~ in} ~ \overline{\mathbb{R}}^+.
 \ees
 Let $(\underline{u},l)$ be the unique solution of the following problem
 \bess
\left\{\begin{aligned}
&\underline u_t=d_1\dd\int_{0}^{l(t)}J_1(x-y)\underline u(t,y)\dy-d_1j_1(x)\underline u+r_1\underline u(a-2\underline u), && t>0,~0\le x<l(t),\\
&\underline u(t,l(t))=0, && t>0,\\
&l'(t)=\kappa\dd\int_{0}^{l(t)}\int_{l(t)}^{\infty}
J_1(x-y)\underline u(t,x)\dy\dx, && t>0,\\
&l(0)=l_0>0,\;\; \underline u(0,x)=\tilde{u}_0(x)\le u_0(x), &&x\in[0,l_0],
\end{aligned}\right.
 \eess
with $\kappa>0$, $\tilde{u}_0\in C([0,l_0])$ and $\tilde{u}_0(x)>0=\tilde{u}_0(l_0)$ in $[0,l_0)$. A comparison argument yields $u\ge \underline{u}$ in $[0,\yy)\times[0,\yy)$. By \cite[Theorem 4.3]{LW21} and \cite[Theroems 2.2 and 4.1]{LW211}, for large $\kappa>0$,
 \bes\label{2.3}\left\{\begin{aligned}
&\lim_{t\to\yy}\max_{x\in[0,\,ct]}\kk|\underline u(t,x)-a/2\rr|=0 ~ {\rm for ~ } c\in(0,c^{\kappa}_0), ~ ~ {\rm if ~ {\bf(J1)} ~ is ~ satisfied },\\
&\lim_{t\to\yy}\max_{x\in[0,\,ct]}\kk|\underline u(t,x)-a/2\rr|=0 ~ {\rm for ~ } c>0, ~ ~ {\rm if ~ {\bf(J1)} ~ is ~ violated},\\
&\lim_{t\to\yy}\max_{x\in[0,\,s(t)]}\kk|\underline u(t,x)-a/2\rr|=0 ~ {\rm for } ~  s(t)=t^{\frac{1}{\gamma-1}}o(1), ~ ~ {\rm if ~}J_1\approx|x|^{-\gamma} ~ {\rm with ~ } \gamma\in(1,2),\\
&\lim_{t\to\yy}\max_{x\in[0,\,s(t)]}\kk|\underline u(t,x)-a/2\rr|=0 ~ {\rm for } ~ s(t)=(t\ln t)o(1), ~ ~ {\rm if ~}J_1\approx|x|^{-2},
\end{aligned}\right.\ees
where $c^{\kappa}_0$ is uniquely given by \eqref{1.3} but with $(d,J,f,\mu)$ replaced by $(d_1,J_1,r_1\phi(a-2\phi),\kappa)$.

If {\bf(J2)} holds true for $J_1$, by \cite[Theorem 5.1]{DLZ}, we know $\lim_{\kappa\to\yy}c^{\kappa}_0=c_*$. Thus for any $c<c_*$ one can take $\kappa$ large enough such that $c<c^{\kappa}_0$. Together with comparison principle and the first result in \eqref{2.3}, we immediately derive
\bes\label{2.4}
\liminf_{t\to\yy}u(t,x)\ge {a}/{2} ~ ~ {\rm uniformly ~ in ~}[0,ct].\ees

If $J_1$ does not satisfies {\bf(J2)} but meets {\bf(J1)}, we learn from \cite[Theorem 5.2]{DLZ} that $\lim_{\kappa\to\yy}c^{\kappa}_0=\yy$. For any $c>0$, we choose $\kappa\gg 1$ such that $c<c^{\kappa}_0$. Similarly, \eqref{2.4} holds. If $J_1$ violates {\bf(J1)}, we can easily derive \eqref{2.4} by a comparison argument and the second result in \eqref{2.3}. If $J_1\approx|x|^{-\gamma}$ for $\gamma\in(1,2]$, by the last two conclusions in \eqref{2.3} we have $\liminf_{t\to\yy}u(t,x)\ge {a}/{2}$ uniformly in $[0,s(t)]$.
Then due to \eqref{2.2}, we get our desired conclusions. This finishes the proof.
 \end{proof}

\subsection{Spreading case: $h_{\yy}=\yy$}

In this subsection, we discuss the longtime behaviors of solution components $u$ and $v$ of \eqref{1.1} when spreading occurs. The following lemma is crucial to our discussion.
\begin{lemma}\label{l2.1} Assume that $s(t)$ is continuous in $[0,\yy)$ and strictly increasing to $\yy$, $s(0)=s_0>0$ and $P(x)$ satisfies {\bf(J)}. Let $d,\alpha,\beta$ be positive constants, and  $w$ be a unique solution of
\bess
\left\{\begin{aligned}
&w_t=d\dd\int_{0}^{s(t)}P(x-y)w(t,y)\dy-dp(x) w+w(\alpha-\beta w), && t>0,~0\le x<s(t),\\
&w(t,s(t))=0, && t>0,\\
&w(0,x)=w_0(x), &&0\le x\le s_0,
\end{aligned}\right.
 \eess
where $w_0\in C([0,s_0])$, $w_0(x)>0=w_0(s_0)$ in $[0,s_0)$ and $p(x)=\int_{0}^{\yy}P(x-y)\dy$. Then the followings hold true:

\sk{\rm(1)}\, $\lim_{t\to\yy}w(t,x)=\alpha/\beta$ locally uniformly in $\overline{\mathbb{R}}^+$.

\sk{\rm(2)}\, Suppose that $P$ satisfies {\bf(J1)} and $\liminf_{t\to\yy}s(t)/t:=\underline{s}\in(0,\yy]$.  Then
\bess\left\{\begin{aligned}
&\lim_{t\to\yy}\max_{x\in[0,\,ct]}|w(t,x)-{\alpha}/{\beta}|=0 ~ {\rm for ~ any ~ } c\in(0,\min\{\underline{s},C_*\}), ~ ~ {\rm if ~} P ~ {\rm  satisfies ~ {\bf(J2)} },\\
&\lim_{t\to\yy}\max_{x\in[0,\,ct]}|w(t,x)-{\alpha}/{\beta}|=0 ~ {\rm for ~ any ~ } c\in(0,\underline{s}), ~ ~ {\rm if ~} P ~ {\rm  violates ~ {\bf(J2)} },\end{aligned}\right.
\eess
where $C_*$ is the minimal speed of the traveling wave problem \eqref{2.0} with $(d_1,J,r_1\phi(a-2\phi))$ replaced by $(d,P,\phi(\alpha-\beta \phi))$.

\sk{\rm(3)}\, Suppose that $P$ does not satisfy {\bf(J1)} and $\liminf_{t\to\yy}s(t)/t=\underline{s}\in(0,\yy]$. Then
    \[\lim_{t\to\yy}\max_{x\in[0,\,ct]}|w(t,x)-{\alpha}/{\beta}|=0 ~ {\rm for ~ any ~ } c\in(0,\underline{s}).\]

\sk{\rm(4)}\, Suppose that there exist $C_1,C_2>0$,  $\gamma\in(1,2)$ and $\lambda\in(1,1/(\gamma-1)]$ such that $P(x)\ge C_1|x|^{-\gamma}$ for $|x|\gg1$ and $s(t)\ge C_2t^{\lambda}$ for $t\gg1$. Then
 \[\lim_{t\to\yy}\max_{x\in[0,\,t^{\theta}]}|w(t,x)-{\alpha}/{\beta}|=0 ~ {\rm for ~ any ~ } \theta\in(1,\lambda).\]

  \sk{\rm(5)}\, Suppose that there exist $C_1,C_2>0$ and $0<\eta\le1$ such that $P(x)\ge C_1|x|^{-2}$ for $|x|\gg1$ and $s(t)\ge C_2t(\ln t)^{\eta}$ for $t\gg1$.  Then
 \[\lim_{t\to\yy}\max_{x\in[0,\,t[\ln (t+1)]^{\omega}]}|w(t,x)-{\alpha}/{\beta}|=0 ~ {\rm for ~ any ~ } \omega\in(0,\eta).\]
\end{lemma}

\begin{proof} \sk{\rm(1)}\, This proof is similar to that of \cite[Theorem 3.9]{CDLL}, and the details are thus omitted here.

\sk{\rm(2)}\, Firstly, a comparison argument shows that $\limsup_{t\to\yy}w(t,x)\le\alpha/\beta$ uniformly in $\overline{\mathbb{R}}^+$. It is thus enough to show that
 \bes
\liminf_{t\to\yy}w(t,x)\ge\alpha/\beta\;\;\;\text{ uniformly in}\;\;[0,ct].
 \label{x.1}\ees

{\bf Step 1}:\, {\it Proof of the assertion with {\bf(J2)} satisfied by $P$}.

We first assume that $P$ has a compact support, say ${\rm supp} P\subset [-K,K]$ for some $K>0$. By the assumption and conclusion (1), for small $\ep>0$ and $c<c_1<\min\{\underline{s},C_*\}$, there is $T>0$ such that $s(t)>c_1t+2K$ and $w(t,x)\ge (1-\ep)\alpha/{\beta}$ for $t\ge T$ and $x\in[0,2K]$. Then $w$ satisfies
\bess
\left\{\begin{aligned}
&w_t\ge d\dd\int_{0}^{c_1t+2K}P(x-y)w(t,y)\dy-dp(x)w+w(\alpha-\beta w), && t>T,~0\le x\le c_1t+2K,\\
&w(t,c_1t+2K)>0, && t>T,\\
&w(T,x)>0, &&0\le x\le c_1T+2K.
\end{aligned}\right.
 \eess
 Define $\tilde{w}(t,x)=w(t+T,x)$ for $t>0$ and $0\le x\le c_1t+2K$. Then
\bess
\left\{\begin{aligned}
&\tilde w_t\ge d\dd\int_{0}^{c_1t+2K}P(x-y)\tilde w(t,y)\dy-dp(x)\tilde w+\tilde w(\alpha-\beta \tilde w), && t>0,~0\le x\le c_1t+2K,\\
&\tilde w(t,c_1t+2K)>0, && t>0,\\
&\tilde w(0,x)>0, &&0\le x\le 2K.
\end{aligned}\right.
 \eess
Since $P$ satisfies {\bf(J2)}, by \cite[Theorems 1.2 and 5.1]{DLZ}, the semi-wave problem \eqref{1.3} with $(P,\phi(\alpha-\beta \phi),\zeta)$ in place of $(J,f,\mu)$ has a unique solution pair $(c^{\zeta}_0,\phi_{c^{\zeta}_0})$ for any $\zeta>0$ and $\lim_{\zeta\to\yy}c^{\zeta}_0=C_*$. Thus $c_1<\min\{\underline{s},c^{\zeta}_0\}$ by choosing $\zeta$ large sufficiently. For such $\zeta$, we set
 \[\underline{w}=(1-\ep)\phi_{c^{\zeta}_0}(x-c_1t-2K),\;\;t>0,\; 0\le x\le c_1t+2K.\]
We shall check that the following inequalities hold true:
 \bes\label{2.5}
\left\{\begin{aligned}
&\underline w_t\le d\dd\int_{0}^{c_1t+2K}P(x-y)\underline w(t,y)\dy-dp(x)\underline w+\underline w(\alpha-\beta \underline w), && t>0,~K\le x\le c_1t+2K,\\
&\underline w(t,c_1t+2K)=0, && t>0,\\
&\underline w(t,x)\le \tilde{w}(t,x), && t>0, ~ 0\le x\le K,\\
&\underline w(0,x)\le \tilde{w}(0,x), &&0\le x\le 2K.
\end{aligned}\right.
 \ees
Once it is done, then $\underline w(t,x)\le \tilde{w}(t,x)=w(t+T,x)$ for $t>0$ and $0\le x \le c_1t+2K$ by the comparison principle. On the other hand,
  \[\max_{x\in[0,ct]}\big|\underline{w}(t,x)-(1-\ep)\alpha/{\beta}\big|
  =(1-\ep)\left(\alpha/{\beta}-\phi_{c^{\zeta}_0}(ct-c_1t-2K)\right)\to0 ~ {\rm as } ~t\to\yy.\]
Thus, $\liminf_{t\to\yy}w(t,x)\ge (1-\ep){\alpha}/{\beta}$ uniformly in $[0,ct]$. The arbitrariness of $\ep$ implies \eqref{x.1}.

Now we verify \eqref{2.5}. Obviously, $\underline w(t,c_1t+2K)=0$ and $\underline{w}(t,x)\le(1-\ep)\alpha/{\beta}\le w(t+T,x)=\tilde{w}(t,x)$ for $t\ge0$ and $0\le x\le K$, and $\underline w(0,x)\le \tilde{w}(0,x)$ for $0\le x\le 2K$. It thus remains to verify the first inequality of \eqref{2.5}. As $c_1<c^{\zeta}_0$, it follows that, for $t>0$ and $K\le x\le c_1t+2K$,
\bess
\underline{w}_t&=&-(1-\ep)c_1{\phi_{c^{\zeta}_0}}'(x-c_1t-2K)\le-(1-\ep)c^{\zeta}_0{\phi_{c^{\zeta}_0}}'(x-c_1t-2K)\\
&=&(1-\ep)\left(d\dd\int_{-\yy}^{ c_1t+2K}\!\!P(x-y)\phi_{c^{\zeta}_0}(y-c_1t-2K)\dy-d\phi_{c^{\zeta}_0}(x-c_1t-2K)+\phi_{c^{\zeta}_0}(\alpha-\beta \phi_{c^{\zeta}_0})\right)\\
&\le& d\dd\int_{0}^{c_1t+2K}P(x-y)\underline w(t,y)\dy-dp(x)\underline w+ \underline w(\alpha-\beta \underline w).
\eess

Now we turn to the case where $P$ does not have a compact support. Define
\[\xi(x)=1 ~\;\text{for}\; |x|\le1, ~ ~ ~ \xi(x)=2-|x| ~\; \text{for}\; 1\le|x|\le2, ~ ~ \xi(x)=0 ~\;\text{for}\; |x|\ge2,\]
and $P_n(x)=\xi(\frac{x}{n})P(x)$. Clearly, $P_n$ are supported compactly and nondecreasing in $n$, $P_n(x)\le P(x)$, $P_n(x)\to P(x)$ in  $L^1(\mathbb{R})$ and locally uniformly in $\mathbb{R}$ as $n\to\yy$. Direct calculations show that
\[p_n(x):=\int_{-x}^{\yy}P_n(y)\dy\to p(x)=\int_{-x}^{\yy}P(y)\dy ~ ~ {\rm uniformly~in~}\mathbb{R}{\rm ~as~} n\to\yy.\]
Hence, for any $\delta\in(0,\alpha)$, there is $N_{\delta}>0$ such that $d(p_n(x)-p(x))+\delta>0$ for $x\ge0$ and $n\ge N_{\delta}$. Take $\hat{d}_n=d\|P_n\|_{L^1(\mathbb{R})}, \hat{P}_n=P_n\|P_n\|^{-1}_{L^1(\mathbb{R})}$ and $\hat{p}_n=p_n\|P_n\|^{-1}_{L^1(\mathbb{R})}$. Let $w_n$ be the unique solution of
\bess
\left\{\begin{aligned}
&w_{nt}=\hat{d}_n\dd\int_{0}^{s(t)}\hat P_n(x-y)w_n(t,y)\dy-\hat {d}_n\hat{p}_n(x) w_n+w_n(\alpha-\delta-\beta w_n), && t>0,~0\le x<s(t),\\
&w_n(t,s(t))=0, && t>0,\\
&w_n(0,x)=w_0(x), &&0\le x\le s_0.
\end{aligned}\right.
 \eess
A comparison argument shows that $w(t,x)\ge w_n(t,x)$ in $\overline{\mathbb{R}}^+\times [0,s(t)]$ for any $n\ge N_{\delta}$. Clearly, the following semi-wave problem
\bes\label{2.6}\left\{\begin{array}{lll}
 \hat{d}_n\dd\int_{-\yy}^{0}\hat P_n(x-y)\phi(y)\dy-\hat d_n\phi+c\phi'+\phi(\alpha-\delta-\beta \phi)=0, \quad -\yy<x<0,\\[1mm]
\phi(-\yy)=(\alpha-\delta)/{\beta},\ \ \phi(0)=0, \ \ c=\zeta\dd\int_{-\yy}^{0}\int_{0}^{\yy}\hat P_n(x-y)\phi(x)\dy\dx
 \end{array}\right.
 \ees
has a unique solution pair $(c^{\zeta}_{n,\delta},\phi^{\zeta}_{n,\delta})$. From Step 2 in the proof of \cite[Theorem 3.15]{LW21} we know $\lim_{n\to\yy}c^{\zeta}_{n,\delta}=c^{\zeta}_{0,\delta}$ which is determined uniquely by problem \eqref{2.6} with $(\hat{d}_n,\hat P_n)$ replaced by $(d,P)$.
Moreover, since $\lim_{\delta\to0}c^{\zeta}_{0,\delta}=c^{\zeta}_0$ and $\lim_{\zeta\to\yy}c^{\zeta}_0=C_*$, for any $c<c_1<\min\{\underline{s},C_*\}$, there are large $\zeta>0$ and small $\delta_0<\alpha$ such that $c<c_1<\min\{\underline{s},c^{\zeta}_{0,\delta}\}$ for any $\delta\in(0,\delta_0)$. For such $\zeta$ and $\delta$, we can find some $N_1>N_{\delta}$ such that  $c_1<\min\{\underline{s},c^{\zeta}_{N_1,\delta}\}$. Similarly to the above we can deduce $\liminf_{t\to\yy}w_{N_1}(t,x)\ge (\alpha-\delta)/{\beta}$ uniformly in $[0,ct]$. The arbitrariness of $\delta$ implies \eqref{x.1}.

{\bf Step 2}:\, {\it Proof of the result with {\bf(J2)} violated by $P$}.

Similarly, we are going to construct an adequate lower solution by using the solution of a semi-wave problem. Define $w_n$ and $c^{\zeta}_{n,\delta}$ as above. Noting $\lim_{n\to\yy}c^{\zeta}_{n,\delta}=c^{\zeta}_{0,\delta}$ and $\lim_{\zeta\to\yy}c^{\zeta}_{0,\delta}=\yy$, for any $c<c_1<\underline{s}$, there are large $N_1$ and $\zeta$ such that $c<c_1<c^{\zeta}_{N_1,\delta}$. Hence we can define a lower solution $\underline{w}_{N_1}=(1-\ep)\phi^{\zeta}_{N_1,\delta}(x-c_1t-2N_1)$. Then, arguing as in the above analysis, one can easily derive the desired result.

\sk{\rm(3)}\,  It clearly suffices to show \eqref{x.1}. We first deal with the case $\liminf_{t\to\yy}s(t)/t<\yy$. Since $P$ does not satisfies {\bf(J1)}, similarly to the proof of \cite[Proposition 5.1]{DN213} we have $\lim_{n\to\yy}c^{\zeta}_{n,\delta}=\yy$ for any $\delta\in(0,\alpha)$. Set $N_1\gg 1$ such that $\underline{s}<c^{\zeta}_{N_1,\delta}$. As in the proof of conclusion (2), we obtain $\liminf_{t\to\yy}w_{N_1}(t,x)\ge (\alpha-\delta)/{\beta}$ uniformly in $[0,ct]$ for all $c\in(0,\underline{s})$. As before, \eqref{x.1} holds.

 Now we handle the case $\lim_{t\to\yy}s(t)/t=\yy$. Obviously, for any $c_1>0$ there exists $T>0$ such that $s(t)>c_1t+s_0$ for $t\ge T$. Define $\tilde{w}(t,x)=w(t+T,x)$ with $t\ge0$ and $0\le x\le c_1t+s_0$. Then,
 \bess
\left\{\begin{aligned}
&\tilde w_{t}\ge d\dd\int_{0}^{c_1t+s_0}P(x-y)\tilde w(t,y)\dy-d p(x)\tilde w+\tilde w(\alpha-\beta \tilde w), && t>0,~0\le x\le c_1t+s_0,\\
&\tilde w(t,c_1t+s_0)>0, && t>0,\\
&\tilde w(0,x)>0, &&0\le x\le s_0.
\end{aligned}\right.
 \eess
 Let $\underline{w}$ be a solution of
 \bess
\left\{\begin{aligned}
&\underline w_{t}=d\dd\int_{0}^{c_1t+s_0} P(x-y)\underline w(t,y)\dy-dp(x)\underline w+\underline w(\alpha-\beta \underline w), && t>0,~0\le x\le c_1t+s_0,\\
&\underline w(t,c_1t+s_0)=0, && t>0,\\
&\underline w(0,x)=\underline{w}_0(x)\le\tilde w(0,x), &&0\le x\le s_0.
\end{aligned}\right.
 \eess
By the above conclusion with $\underline{s}\in(0,\yy)$ and a comparison argument, we get  \eqref{x.1}.

\sk{\rm(4)}\, Take $\theta_1\in(\theta,\lambda)$. We shall verify that there exist constants  $\ep,K_1,K_2, T_0>0$ such that functions defined by $l(t)=(K_1t+K_2)^{\theta_1}$ and $\underline{w}(t,x)=K_{\ep}\min\{1,2(l(t)-x)/{l(t)}\}$ with $K_{\ep}=\alpha/\beta-\sqrt\ep$ satisfy, for any $T>T_0$,
   \bes\label{2.7}
\left\{\begin{aligned}
&\underline w_t\le d\dd\int_{0}^{l(t)}P(x-y)\underline w(t,y)\dy-dp(x)\underline w+ \underline w(\alpha-\beta\underline{w}), && t>0,~x\in[0,l(t))\setminus\left\{{l(t)}/{2}\right\},\\
&\underline w(t,l(t))\le0,&& t>0,\\
&\underline w(0,x)\le w(T,x),&&x\in[0,K_2^{\theta_1}].
\end{aligned}\right.
 \ees
Once \eqref{2.7} is obtained, as $\lambda>\theta_1$ and $s(t)\ge Ct^{\lambda}$ for $t\gg1$, we can find $T_1>T_0$ such that $s(t)>l(t)$ for $t\ge T_1$. Then the function  $\tilde{w}(t,x)=w(t+T_1,x)$ satisfies
\bess
\left\{\begin{aligned}
&\tilde w_t\ge d\dd\int_{0}^{l(t)}P(x-y)\tilde w(t,y)\dy-dp(x)\tilde w+\tilde w(\alpha-\beta \tilde w), && t>0,~0\le x\le l(t),\\
&\tilde w(t,l(t))>0, && t>0,\\
&\tilde w(0,x)=w(T_1,x)\ge\underline{w}(0,x), &&0\le x\le K_2^{\theta_1}.
\end{aligned}\right.
 \eess
The comparison principle indicates $\tilde{w}(t,x)\ge\underline{w}(t,x)$ for $t\ge0$ and $0\le x\le l(t)$. Since
 \bess
\max_{x\in[0,\,t^{\theta}]}|\underline w(t,x)- {\alpha}/{\beta}+\sqrt{\ep}|=({\alpha}/{\beta}-\sqrt{\ep})\left(1-\min\left\{1,\,2(l(t)-t^{\theta})/{l(t)}\right\}\right)\to0 ~ ~ {\rm as} ~ t\to\yy,
\eess
it follows that $\liminf_{t\to\yy}w(t,x)\ge{\alpha}/{\beta}-\sqrt{\ep}$ uniformly in $[0,t^{\theta}]$.
Thanks to the arbitrariness of $\ep$, we easily derive the desired result.

Thus it remains to show that \eqref{2.7} holds. Although the following analysis as well as that of conclusion (5) are similar to those of \cite[Theorem 4.1]{LW211}, we give the details for the convenience of readers. Beginning with proving the first inequality of \eqref{2.7}, we easily see $\underline{w}(t,x)\ge K_{\ep}(1-{x}/{l(t)})$ for $t\ge0$ and $x\in[0,l(t)]$. For $x\in [{l(t)}/{4},l(t)]$, we have that, when $K_2$ is large,
\bess
\int_{0}^{l(t)}P(x-y)\underline w(t,y)\dy&=&\int_{-x}^{l(t)-x}P(y)\underline w(t,x+y)\dy\\
 &\ge&\int_{-{l(t)}/{4}}^{-{l(t)}/{8}}P(y)\underline w(t,x+y)\dy\\
 &\ge& K_{\ep}\int_{-{l(t)}/{4}}^{-{l(t)}/{8}}\frac{C_1}{|y|^{\gamma}}
 \left(1-\frac{x+y}{l(t)}\right)\dy\\
 &\ge&\frac{K_{\ep}}{l(t)}\int_{-{l(t)}/{4}}^{-{l(t)}/{8}}\frac{C_1}{|y|^{\gamma}}(-y)\dy\ge\tilde C_1K_{\ep}l^{1-\gamma}(t),
  \eess
where $\tilde{C_1}$ depends only on $P$. For $x\in [0,{l(t)}/{4}]$, we have
\bess
\int_{0}^{l(t)}P(x-y)\underline w(t,y)\dy&\ge&\int_{0}^{{l(t)}/{2}}P(x-y)\underline w(t,y)\dy\\
&=&K_{\ep}\int_{0}^{{l(t)}/{2}}P(x-y)\dy\\
&\ge& K_{\ep}(p(x)-\ep)+K_{\ep}\left(\ep-\int_{{l(t)}/{4}}^{\yy}P(y)\dy\right)\\
&\ge&  K_{\ep}(p(x)-\ep)= (p(x)-\ep)\underline{w}(t,x)
\eess
provided that $K_2\gg1$. Moreover, for $x\in [0,{l(t)}/{4}]$,
\bes\label{2.8}
\underline w(\alpha-\beta\underline{w})\ge \frac{\alpha}{2}\min\{\underline w,\,\frac{\alpha}{\beta}-\underline w\}\ge \frac{\alpha}{2}\min\{K_{\ep},\,\frac{\alpha}{\beta}-\underline w\}\ge \frac{\alpha}{2}\sqrt{\ep} ~ \;{\rm ~ if ~ } \ep\le \frac{\alpha^2}{4\beta^2}.
\ees
From \cite[Proposition 4.3]{LW211}, by sending $L_2=l(t)$ and $L_1=l(t)/2$ with $K_2\gg1$, we obtain
\[\int_{0}^{l(t)}P(x-y)\underline w(t,y)\dy\ge(1-\ep^2)\underline{w}(t,x) ~ ~ ~ {\rm for} ~ t>0, ~ x\in[{l(t)}/{4},l(t)].\]
Additionally, it is easy to see
\bes\label{2.9}
\underline w(\alpha-\beta\underline{w})\ge \frac{\alpha}{2}\min\{\underline w,\,\frac{\alpha}{\beta}-\underline w\}\ge \frac{\alpha}{2}\ep\underline{w} ~ ~ {\rm if } ~ \ep\le \min\{1,\left(\frac{\beta}{\alpha}\right)^2\}.
\ees
Therefore, we can obtain that for $x\in [0,{l(t)}/{4}]$,
\bess
d\dd\int_{0}^{l(t)}P(x-y)\underline w(t,y)\dy-dp(x)\underline w+ \underline w(\alpha-\beta\underline{w})\ge-d\ep\underline{w}+\frac{\alpha}{2}\sqrt{\ep}
\ge-d\ep \frac{\alpha}{\beta}+\frac{\alpha}{2}\sqrt{\ep}\ge0\eess
if $\ep\le(\frac{\beta}{2d})^2$,
and for $x\in [{l(t)}/{4},l(t)]$,
\bess
&&d\dd\int_{0}^{l(t)}P(x-y)\underline w(t,y)\dy-dp(x)\underline w+\underline w(\alpha-\beta\underline{w})\\
&&\ge d\dd\int_{0}^{l(t)}P(x-y)\underline w(t,y)\dy-\left(d-\frac{\alpha}{2}\ep\right)\underline w(t,x)\\
&&=\left[\left(\min\left\{\frac{\alpha\ep}{4},d\right\}+\left\{d-\frac{\alpha\ep}{4}\right\}^+\right)\int_{0}^{l(t)}P(x-y)\underline w(t,y)\dy\right]-\left(d-\frac{\alpha\ep}{2}\right)\underline w(t,x)\\
&&\ge\min\left\{\frac{\alpha\ep}{4},d\right\}\tilde C_1 K_{\ep}l^{1-\gamma}(t)+\left\{d-\frac{\alpha\ep}{4}\right\}^+(1-\ep^2)\underline{w}(t,x)-\left(d-\frac{\alpha\ep}{2}\right)\underline w(t,x)\\
 &&\ge\min\left\{\frac{\alpha\ep}{4},d\right\}\tilde C_1 K_{\ep}l^{1-\gamma}(t)\eess
if $\ep\le\frac{\alpha}{4d}$. On the other hand, we have $\underline{w}_t(t,x)=0$ for $t>0$ and $x\in[0,{l(t)}/{2}]$. For $t>0$ and $x\in({l(t)}/{2},l(t))$, noticing $\theta_1<{1}/(\gamma-1)$ and $K_2\gg1$, we obtain
\[\underline{w}_t(t,x)=2K_{\ep}\frac{xl'(t)}{l^2(t)}\le\frac{2\theta_1 K_1K_{\ep}}{K_1t+K_2}\le\min\left\{\frac{\alpha\ep}{4},d\right\}\frac{\tilde C_1 K_{\ep}}{(K_1t+K_2)^{\theta_1(\gamma-1)}} ~ ~ {\rm if } ~ K_1\le \frac{\tilde C_1\min\left\{\frac{\alpha\ep}{4},d\right\}}{2\theta_1}.\]
The first inequality of \eqref{2.7} holds. The second inequality is obvious. Now we focus on the last one. For $\ep$ and $K_2$ as chosen above, by conclusion (1), there is $T_0>0$ such that $\underline{w}(0,x)\le K_{\ep}=\alpha/\beta-\sqrt{\ep}\le w(T,x)$ for $T\ge T_0$ and $0\le x\le K_2^{\theta_1}$.
Thus conclusion (4) follows.

\sk{\rm(5)}\, For $0<\omega<\omega_1<\eta$ and small $\ep>0$, we define $l(t)=K_1(t+K_2)\left[\ln(t+K_2)\right]^{\omega_1} $ and $\underline{w}(t,x)=K_{\ep}\min\big\{1,(l(t)-x)/{(t+K_2)^{1/2}}\big\}$ with $K_{\ep}={\alpha}/{\beta}-\sqrt{\ep}$ and $K_1,K_2>0$. We prove that there are $K_1,K_2$ and $T_0>0$ such that
  \bes\label{2.10}
\left\{\begin{aligned}
&\underline w_t\le d\dd\int_{0}^{l(t)}P(x-y)\underline w(t,y)\dy-dp(x)\underline w+ \underline w(\alpha-\beta\underline{w}), \\
&\hspace{35mm} t>0,~x\in[0,l(t))\setminus\big\{l(t)-(t+K_2)^{{1}/{2}}\big\},\\
&\underline w(t,l(t))=0,\hspace{12mm} t>0,\\
&\underline w(0,x)\le w(T,x),\hspace{4mm} T>T_0, ~ x\in[0,K_1K_2\left(\ln K_2\right)^{\omega_1}].
\end{aligned}\right.
\ees
Once \eqref{2.10} is obtained, by the assumption on $s(t)$ we can find a $T_1>T_0$ such that $s(t)>l(t)$ for $t\ge T_1$. We again define $\tilde{w}(t,x)=w(t+T_1,x)$, and then $\tilde{w}$ satisfies
\bess
\left\{\begin{aligned}
&\tilde w_t\ge d\dd\int_{0}^{l(t)}P(x-y)\tilde w(t,y)\dy-dp(x)\tilde w+\tilde w(\alpha-\beta \tilde w), && t>0,~0\le x\le l(t),\\
&\tilde w(t,l(t))>0, && t>0,\\
&\tilde w(0,x)=w(T_1,x)\ge\underline{w}(0,x), &&0\le x\le K_1K_2\left(\ln K_2\right)^{\omega_1}.
\end{aligned}\right.
 \eess
 By comparison principle, $\tilde{w}(t,x)\ge\underline{w}(t,x)$ for $t\ge0$ and $0\le x\le l(t)$. Moreover, as $t\to\yy$,
 \bess
\max_{x\in[0,\,t[\ln (t+1)]^{\omega}]}|\underline w(t,x)- {\alpha}/{\beta}+\sqrt{\ep}|=({\alpha}/{\beta}-\sqrt{\ep})\left(1-\min\left\{1,\,2\left(l(t)-t[\ln (t+1)]^{\omega}\right)/{l(t)}\right\}\right)\to0,
\eess
which, as well as the arbitrariness of $\ep$, implies $\liminf_{t\to\yy}w(t,x)\ge{\alpha}/{\beta}$  uniformly in $[0,t[\ln (t+1)]^{\omega}]$.

Now we are ready to prove the first inequality of \eqref{2.10}. Clearly, $l(t)(t+K_2)^{-1/2}\to\yy$ uniformly in $t\ge0$ as $K_2\to0$, and $\underline{w}(t,x)\ge K_{\ep}\frac{l(t)-x}{2(t+K_2)^{1/2}}$ for $x\in[l(t)-2(t+K_2)^{1/2},l(t)]$. Thus when $K_2$ is large enough, we see that, for $x\in[l(t)-(t+K_2)^{1/2},l(t)]$,
\bess
\int_{0}^{l(t)}P(x-y)\underline w(t,y)\dy&=&\int_{-x}^{l(t)-x}P(y)\underline w(t,x+y)\dy\\
 &\ge&\frac{K_{\ep}}{2}\int_{-(t+K_2)^{1/2}}^{-(t+K_2)^{1/4}}P(y)\frac{l(t)-x-y}{(t+K_2)^{1/2}}\dy\\
&\ge&\frac{K_{\ep}}{2}\int_{-(t+K_2)^{1/2}}^{-(t+K_2)^{1/4}}P(y)\frac{-y}{(t+K_2)^{1/2}}\dy\\
 &\ge&\frac{K_{\ep}C_1}{2}\int_{-(t+K_2)^{1/2}}^{-(t+K_2)^{1/4}}\frac{(-y)^{-1}}{(t+K_2)^{1/2}}\dy=\frac{K_{\ep}C_1\ln(t+K_2)}{8(t+K_2)^{1/2}}.
\eess
Similarly, we also can obtain that, for $x\in[\left(l(t)-(t+K_2)^{1/2}\right)/{2},l(t)-(t+K_2)^{1/2}]$,
\bess
\int_{0}^{l(t)}P(x-y)\underline w(t,y)\dy&=&\int_{-x}^{l(t)-x}P(y)\underline w(t,x+y)\dy\\
 &\ge& K_{\ep}\int_{-(t+K_2)^{1/2}}^{-(t+K_2)^{1/4}}P(y)\dy\\
 &\ge&K_{\ep}C_1\int_{-(t+K_2)^{1/2}}^{-(t+K_2)^{1/4}}(-y)^{-2}\dy\ge\frac{K_{\ep}C_1\ln(t+K_2)}{4(t+K_2)^{1/2}}.
\eess
Moreover, it can be seen from \cite[Proposition 4.3]{LW211} with $L_2=l(t)$ and $L_1=(t+K_2)^{1/2}$ that
\[\int_{0}^{l(t)}P(x-y)\underline w(t,y)\dy\ge(1-\ep^2)\underline w(t,x) ~ ~ ~ {\rm for} ~ t>0, ~ x\in\kk[\left(l(t)-(t+K_2)^{1/2}\right)/{2},l(t)\rr].\]
By similar arguments in the proof of conclusion (4) we have that, for $x\in[0,\left(l(t)-(t+K_2)^{1/2}\right)/{2}]$,
\[\int_{0}^{l(t)}P(x-y)\underline w(t,y)\dy\ge(p(x)-\ep)\underline{w}(t,x),\]
which, together with \eqref{2.8}, leads to
\[d\dd\int_{0}^{l(t)}P(x-y)\underline w(t,y)\dy-dp(x)\underline w+ \underline w(\alpha-\beta\underline{w})\ge-d\ep \frac{\alpha}{\beta}+\frac{\alpha}{2}\sqrt{\ep}\ge0 ~ ~ {\rm if } ~ \ep\le\left(\frac{\beta}{2d}\right)^2.\]
For $x\in[\left(l(t)-(t+K_2)^{1/2}\right)/{2},l(t)]$, we have
\bess
&&d\dd\int_{0}^{l(t)}P(x-y)\underline w(t,y)\dy-dp(x)\underline w+ \underline w(\alpha-\beta\underline{w})\\
&&\ge d\dd\int_{0}^{l(t)}P(x-y)\underline w(t,y)\dy-\left(d-\frac{\alpha\ep}{2}\right)\underline w(t,x)\\
&&=\left[\left(\min\left\{\frac{\alpha\ep}{4},d\right\}+\left\{d-\frac{\alpha\ep}{4}\right\}^+\right)\int_{0}^{l(t)}P(x-y)\underline w(t,y)\dy\right]-\left(d-\frac{\alpha\ep}{2}\right)\underline w(t,x)\\
&&\ge\min\left\{\frac{\alpha\ep}{4},d\right\}\frac{K_{\ep}C_1\ln(t+K_2)}{8(t+K_2)^{1/2}}+\left\{d-\frac{\alpha\ep}{4}\right\}^+(1-\ep^2)\underline{w}(t,x)-\left(d-\frac{\alpha\ep}{2}\right)\underline w(t,x)\\
&&\ge\min\left\{\frac{\alpha\ep}{4},d\right\}\frac{K_{\ep}C_1\ln(t+K_2)}{8(t+K_2)^{1/2}}\; ~ ~ {\rm if } ~ \ep\le\frac{\alpha}{4d}.
\eess
On the other hand, we have $\underline{w}_t(t,x)=0$ for $t>0$ and $x\in[0,l(t)-(t+K_2)^{1/2})$, and for $x\in(l(t)-(t+K_2)^{1/2},l(t)]$,
\[\underline{w}_t(t,x)\le\frac{\ln(t+K_2)}{(t+K_2)^{1/2}}(K_1K_{\ep}+K_1K_{\ep}\omega_1)\le\min\left\{\frac{\alpha\ep}{4},d\right\}\frac{K_{\ep}C_1\ln(t+K_2)}{8(t+K_2)^{1/2}}.\]
 Thus we obtain the first inequality of \eqref{2.10} if $K_1\le \frac{C_1\min\left\{\frac{\alpha\ep}{4},d\right\}}{8(1+\omega_1)}$. Moreover, for $K_1$ and $K_2$ as above, by conclusion (1) we can choose $T_0>0$ such that $\underline{w}(0,x)\le K_{\ep}=\alpha/\beta-\sqrt{\ep}\le w(T,x)$ with $T\ge T_0$ and $0\le x\le K_1K_2\left(\ln K_2\right)^{\omega_1}$. Therefore, all assertions are proved.
\end{proof}
 Let $(u^*,v^*)$ be the unique positive root of $f_1(u,v)=0$ and $ f_2(u,v)=0$.
By \cite[Theorem 1.2]{DLZ}, the semi-wave problem \eqref{1.3}, where $(d,J,f)$ is replaced with $(d_2,J_2,r_2\phi(1-2\phi))$,
 has a unique solution pair $(c_0,\phi_0)$ with $c_0>0$ and $\phi_0$ nonincreasing in $(-\yy,0]$ if and only if $J_2$ satisfies {\bf(J1)}.
 \begin{theorem} \label{t2.3}If $h_{\yy}=\yy$, then we have the following conclusions:

 \sk{\rm(1)}\, $\lim_{t\to\yy}u(t,x)=u^*$ and $\lim_{t\to\yy}v(t,x)=v^*$ locally uniformly in $\overline{\mathbb{R}}^+$.

 \sk{\rm(2)}\, If $J_1$ satisfies {\bf(J2)} and $J_2$ satisfies {\bf(J1)}, then
 \[\lim_{t\to\yy}\max_{x\in[0,\,ct]}\left\{|u(t,x)-u^*|+|v(t,x)-v^*|\right\}=0 ~ {\rm for ~ any ~ } c\in(0,\min\{c_*,c_0\}).\]

 \sk{\rm(3)}\, If $J_1$ violates {\bf(J2)} and $J_2$ satisfies {\bf(J1)}, then
 \[\lim_{t\to\yy}\max_{x\in[0,\,ct]}\left\{|u(t,x)-u^*|+|v(t,x)-v^*|\right\}=0 ~ {\rm for ~ any ~ } c\in(0,c_0).\]

  \sk{\rm(4)}\, If $J_1$ meets {\bf(J2)} and $J_2$ violates {\bf(J1)}, then
 \[\lim_{t\to\yy}\max_{x\in[0,\,ct]}\left\{|u(t,x)-u^*|+|v(t,x)-v^*|\right\}=0 ~ {\rm for ~ any ~ } c\in(0,c_*).\]

  \sk{\rm(5)}\, If $J_1$ violates {\bf(J2)} and {\bf(J1)} does not hold for $J_2$, then
 \[\lim_{t\to\yy}\max_{x\in[0,\,ct]}\left\{|u(t,x)-u^*|+|v(t,x)-v^*|\right\}=0 ~ {\rm for ~ any ~ } c>0.\]

 \sk{\rm(6)}\, If there is $C>0$ such that $\min\{J_1(x),J_2(x)\}\ge C|x|^{-\gamma}$ for $|x|\gg1$ and $\gamma\in(1,2)$. Then
 \[\lim_{t\to\yy}\max_{x\in[0,\,t^\theta]}\left\{|u(t,x)-u^*|+|v(t,x)-v^*|\right\}=0 ~ {\rm for ~ any ~ } 1<\theta<1/(\gamma-1).\]

 \sk{\rm(7)}\, If there is $C>0$ such that $\min\{J_1(x),J_2(x)\}\ge C|x|^{-2}$ for $|x|\gg1$. Then
 \[\lim_{t\to\yy}\max_{x\in[0,\,t[\ln (t+1)]^{\omega}]}\left\{|u(t,x)-u^*|+|v(t,x)-v^*|\right\}=0 ~ {\rm for ~ any ~ } 0<\omega<1.\]
 \end{theorem}
\begin{proof}
\sk{\rm(1)}\, This conclusion can be proved by arguing as in the proof of \cite[Theorem 4.11]{LSW} with some obvious modifications, and thus the details are omitted.

\sk{\rm(2)}\, Let $(U,V)$ be the unique solution of problem
  \[U_t=f_1(U,V),\ \ V_t=f_2(U,V); \ \ \ U(0)=K_1,\ \ V(0)=K_2.\]
The comparison principle implies $u(t,x)\le U(t)$ and $v(t,x)\le V(t)$ for $t\ge0$ and $x\ge0$. It follows from a simple phase-plane analysis that $\lim_{t\to\yy}U(t)=u^*$ and $\lim_{t\to\yy}V(t)=v^*$. Thus
\bes\label{2.11}\limsup_{t\to\yy}u(t,x)\le u^* ~ {\rm and} ~ \limsup_{t\to\yy}v(t,x)\le v^* ~ {\rm uniformly ~ in ~ }\overline{\mathbb{R}}^+.\ees

Since $r_1u(a-u-\frac{u}{1+bv})\ge r_1u(a-2u)$, we can derive from comparison principle and \cite[Theorems 4.3 and 4.4]{LW21}
that $\liminf_{t\to\yy}h(t)/t\ge c_0$. For any $c\in(0,\min\{c_*,c_0\})$, we choose a strictly decreasing sequence $\{c_n\}$ with $c_n\in(c,\min\{c_*,c_0\})$. Clearly, $u$ satisfies
\bess\left\{\begin{aligned}
&u_t\ge d_1\int_{0}^{h(t)}J_1(x-y)u(t,y)\dy-d_1j_1(x)u+r_1u(a-2u), & &t>0,~0\le x<h(t),\\
&u(t,h(t))>0,& &t>0,\\
&u(0,x)=u_0(x)>0,& &0\le x\le h_0.
 \end{aligned}\right.
 \eess
By a comparison consideration and the conclusion (2) of Lemma \ref{l2.1}, we obtain $\liminf_{t\to\yy}u(t,x)\ge a/2:=\underline{u}_1$ uniformly in $[0,c_1t]$.
For small $\ep>0$, there is $T>0$ such that $h(t)>c_2t+h_0$ and $u(t,x)\ge \underline{u}_1-\ep$ for $t\ge T$ and $x\in[0,c_2t+h_0]$. Hence $v$ satisfies
\bess\left\{\begin{aligned}
&v_t\ge d_2\int_{0}^{c_2t+h_0}\!\!J_2(x-y)v(t,y)\dy-d_2j_2(x)v+f_2(\underline{u}_1-\ep,\,v), & &t>T, ~ x\in[0, c_2t+h_0],\\
&v(t,c_2t+h_0)>0,& &t>T,\\
&v(T,x)>0,& &x\in[0,c_2T+h_0].
 \end{aligned}\right.
 \eess
Define $\tilde{v}(t,x)=v(t+T,x)$ with $t\ge0$ and $0\le x\le c_2(t+T)+h_0$. Then
  \bess\left\{\begin{aligned}
&\tilde v_t\ge d_2\int_{0}^{c_2t+h_0}\!\!J_2(x-y)\tilde v(t,y)\dy-d_2j_2(x)\tilde v
+f_2(\underline{u}_1-\ep,\,\tilde v), & &t>0, ~ x\in[0, c_2t+h_0],\\
&\tilde v(t,c_2t+h_0)>0,& &t>0,\\
&\tilde v(0,x)>0,& &x\in[0,h_0].\\
 \end{aligned}\right.
 \eess
Again by the conclusion (2) of Lemma \ref{l2.1} and comparison argument we get
$\liminf_{t\to\yy}v(t,x)\ge \frac{1+c(\underline{u}_1-\ep)}{2+c(\underline{u}_1-\ep)}$ uniformly  in  $[0,c_3t]$. The arbitrariness of $\ep$ implies $\liminf_{t\to\yy}v(t,x)\ge \frac{1+c\underline{u}_1}{2+c\underline{u}_1}:=\underline{v}_1$ uniformly in $[0,c_3t]$.
 For small $\ep>0$, there is $T>0$ such that $v(t,x)\ge \underline{v}_1-\ep$ in $[T,\yy)\times[0,c_4t+h_0]$.
 Then $u$ satisfies
 \bess\left\{\begin{aligned}
&u_t\ge d_1\int_{0}^{c_4t+h_0}\!\!J_1(x-y)u(t,y)\dy-d_1j_1(x)u+f_1(u,\,\underline{v}_1-\ep), & &t>T, ~ x\in[0, c_4t+h_0],\\
&u(t,c_4t+h_0)>0,& &t>T,\\
&u(T,x)>0,& &x\in[0,c_4T+h_0].
 \end{aligned}\right.
 \eess
Similar to the above we can deduce $\liminf_{t\to\yy}u(t,x)\ge \frac{a(1+b\underline{v}_1)}{2+b\underline{v}_1}:=\underline{u}_2$ uniformly in $[0,c_5t]$.
Repeating the above arguments we can find two sequences $\{\underline{u}_n\}$ and $\{\underline{v}_n\}$ such that $\liminf_{t\to\yy}u(t,x)\ge \underline{u}_n$, $\liminf_{t\to\yy}v(t,x)\ge \underline{v}_n$ uniformly in $[0,ct]$, and
$\{\underline{u}_n\}$, $\{\underline{v}_n\}$ satisfy
  \[\underline{u}_1=\frac{a}{2}, ~\; \underline{v}_n=\frac{1+c\underline{u}_n}{2+c\underline{u}_n}, ~ \; \underline{u}_{n+1}=\frac{a(1+b\underline{v}_n)}{2+b\underline{v}_n},\;\;\;n=1,2,\cdots.\]
  Moreover, it is easy to show that $\underline{u}_n$ and $\underline{v}_n$ are both increasing in $n$ and $\underline{u}_n\le a$ and $\underline{v}_n\le1$. Thus we can define $\underline{u}_\yy=\lim_{n\to\yy}\underline{u}_n$ and $\underline{v}_\yy=\lim_{n\to\yy}\underline{v}_n$ with
  \[\underline{v}_\yy=\frac{1+c\underline{u}_\yy}{2+c\underline{u}_\yy}, ~ \;\; \underline{u}_{\yy}=\frac{a(1+b\underline{v}_\yy)}{2+b\underline{v}_\yy}.\]
  By the uniqueness of positive root, we have $\underline{u}_\yy=u^*$ and $\underline{v}_\yy=v^*$. So assertion (2) is proved.

   \sk{\rm(3)}\, Similarly to the above analysis we can make use of conclusions (2) and (3) of Lemma \ref{l2.1} to construct the same sequences $\{\underline{u}_n\}$ and $\{\underline{v}_n\}$, so that for any $c\in(0,c_0)$, $\liminf_{t\to\yy}u(t,x)\ge \underline{u}_n$ and $\liminf_{t\to\yy}v(t,x)\ge \underline{v}_n$ uniformly in $[0,ct]$.
   By \eqref{2.11} and the above discussion on $\{\underline{u}_n\}$ and $\{\underline{v}_n\}$, we complete the proof of conclusion (3).

   \sk{\rm(4)}\ and \sk{\rm(5)}\, Consider the following problem
    \bes\label{2.14}
\left\{\begin{aligned}
&w_t=d_2\dd\int_{0}^{r(t)}J_2(x-y)w(t,y)\dy-d_2j_2(x) w+r_2w(1-2w), && t>0,~0\le x<r(t),\\
&w(t,r(t))=0, && t>0,\\
&r'(t)=\mu\dd\int_{0}^{r(t)}\!\!\int_{r(t)}^{\infty}
J_2(x-y)w(t,x)\dy\dx, && t>0,\\
&r(0)=h(T),\;\; w(0,x)=v(T,x), &&0\le x\le r(0).
\end{aligned}\right.
 \ees
By \cite[Theorem 4.3]{LW21}, spreading happens for $(w,r)$ if $T$ is large enough. Note that $J_2$ violates {\bf(J1)}. Using \cite[Theorem 4.4]{LW21} we know $\lim_{t\to\yy}r(t)/t=\yy$. Moreover, by $f_2(u,v)\ge r_2v(1-2v)$ and the comparison principle, one has $h(t+T)\ge r(t)$ for $t\ge0$, and thus $\lim_{t\to\yy}h(t)/t=\yy$. Then by following the similar lines as above we can finish the proof, and the details are omitted here.

\sk{\rm(6)}\, For any $\theta\in(1,1/(\gamma-1))$, we choose a strictly decreasing sequence $\theta_n\in(\theta,1/(\gamma-1))$. Then $u$ satisfies
 \bess\left\{\begin{aligned}
&u_t\ge d_1\int_{0}^{t^{\theta_1}+h_0}J_1(x-y)u(t,y)\dy-d_1j_1(x)u+r_1u(a-2u), & &t>0,~0\le x<t^{\theta_1}+h_0,\\
&u(t,t^{\theta_1}+h_0)>0,& &t>0,\\
&u(0,x)>0,& &0\le x\le h_0.\\
 \end{aligned}\right.
 \eess
 By the conclusion (4) in Lemma \ref{l2.1} and a comparison argument, we have $\liminf_{t\to\yy}u(t,x)\ge \underline{u}_1$ uniformly in $[0,t^{\theta_2}]$. Moreover, it follows from the assumption on $J_2$ and \cite[Lemma 4.4]{LW21} that $r(t)\ge C t^{\frac{1}{\gamma-1}}$ for $t\gg1$ and some $C>0$. Due to $h(t+T)\ge r(t)$, there is $\tilde C>0$ such that $h(t)\ge \tilde C t^{\frac{1}{\gamma-1}}$ for $t\gg1$. So for any small $\ep>0$, there exists $T_1>0$ such that $h(t)>t^{\theta_3}+h_0$ and $u(t,x)\ge \underline{u}_1-\ep$ for $t\ge T_1$ and $0\le x\le t^{\theta_3}+h_0$. Then $\hat v(t,x)=v(t+T_1,x)$ satisfies
 \bess\left\{\begin{aligned}
&\hat v_t\ge d_2\int_{0}^{t^{\theta_3}+h_0}\!\!J_2(x-y)\hat v(t,y)\dy-d_2j_2(x)\hat v
+f_2(\underline{u}_1-\ep,\,\hat v), & &\!t>0, ~ x\in[0, t^{\theta_3}+h_0],\\
&\hat v(t,t^{\theta_3}+h_0)>0,& &t>0,\\
&\hat v(0,x)>0,& &x\in[0,h_0].\\
 \end{aligned}\right.
 \eess
Using conclusion (4) of Lemma \ref{l2.1}, the comparison principle and arbitrariness of $\ep$, we obtain $\liminf_{t\to\yy}v(t,x)\ge \underline{v}_1$ uniformly in $[0,t^{\theta_4}]$. Repeating the above iteration process, we similarly can derive that $\liminf_{t\to\yy}u(t,x)\ge u^*$ and $\liminf_{t\to\yy}v(t,x)\ge v^*$ uniformly in $[0,t^\theta]$. In view of \eqref{2.11}, we prove this conclusion.

\sk{\rm(7)}\, This result can be proved by following the analogous arguments in the above discussion, and the details are omitted here.
\end{proof}

\section{Criteria for spreading and vanishing}
In this section, we investigate the criteria for spreading and vanishing. The following comparison principle will be used to prove our results.
In this section, let $(u,v,h)$ be a solution of \eqref{1.1}.
\begin{lemma}[Comparison principle]\label{l3.1} Assume that $\bar{u},\bar{u}_t\in C([0,T]\times\overline{\mathbb{R}}^+)$, $\bar{u}\in L^{\yy}([0,T]\times\mathbb{R}^+)$, $\bar{v},\bar{v}_t\in C([0,T]\times[0,\bar{h}(t)])$ and $\bar{h}\in C^1([0,T])$. If $(\bar{u},\bar{v},\bar{h})$ satisfies
\bes\left\{\begin{aligned}\label{3.1}
&\bar u_t\ge d_1\int_{0}^{\yy}\!J_1(x-y)\bar u(t,y)\dy-d_1j_1(x)\bar u+f_1(\bar u, \bar v), & &t\in(0,T],~x\in\overline{\mathbb{R}}^+,\\
&\bar v_t\ge d_2\int_{0}^{\bar h(t)}\!J_2(x-y)\bar v(t,y)\dy-d_2j_2(x)\bar v+f_2(\bar u, \bar v), & &t\in(0,T],~x\in[0,\bar h(t)),\\
&\bar v(t,x)\ge0,& &t\in(0,T], ~x\in[\bar h(t),\yy) \\
&\bar h'(t)\ge\mu\int_{0}^{\bar h(t)}\int_{\bar h(t)}^{\infty}
J_2(x-y)\bar v(t,x)\dy\dx,& &t\in(0,T],\\
&\bar u(0,x)\ge u_0(x), ~ x\in\overline{\mathbb{R}}^+; ~ ~ \bar h(0)\ge h_0>0, ~ \bar v(0,x)\ge v_0(x)& &x\in[0,\bar h(0)],
 \end{aligned}\right.
 \ees
 then we have
 \[\bar{h}(t)\ge h(t), ~ t\in[0,T], ~ ~ ~ \bar{u}(t,x)\ge u(t,x) ~ ~ \bar{v}(t,x)\ge v(t,x), ~ (t,x)\in[0,T]\times\mathbb{R}^+.\]
\end{lemma}
\begin{proof}We first claim that $\bar{u}\ge0$ and $\bar{v}\ge0$ in $[0,T]\times\overline{\mathbb{R}}^+$. In fact, this is a direct consequence of maximum principle if $1+b\bar{v}>0$ in $[0,T]\times[0,\bar{h}(t)]$. Assume on the contrary that there is $(t_1,x_1)\in (0,T]\times[0,\bar{h}(t_1)]$ such that $1+b\bar{v}(t_1,x_1)\le0$. By virtue of $1+b\bar{v}(0,x)\ge 1$ and the continuity, there exists $(t_*,x_*)\in (0,t_1]\times[0,\bar{h}(t_*)]$ such that $1+b\bar{v}(t_*,x_*)=0$ and $1+b\bar{v}>0$ in $[0,t_*)\times[0,\bar{h}(t)]$. Apply the maximum principle to the equation of $\bar u$ to deduce that $\bar{u}\ge0$ in $[0,s]\times\overline{\mathbb{R}}^+$ for any $s\in(0,t_*)$. Hence, $\bar{u}\ge0$ in $[0,t_*]\times\overline{\mathbb{R}}^+$. Applying the maximum principle to $\bar v$ we have $\bar{v}\ge0$ in $[0,t_*]\times[0,\bar{h}(t)]$. This contradiction implies our claim.

For $0<\ep\ll 1$, we define $u^{\ep}_0(x)=(1-\ep)u_0(x)$, $h^{\ep}_0=(1-\ep)h_0$ and $\mu^{\ep}=(1-\ep)\mu$. Choose $v^{\ep}_0\in C([0,h^{\ep}_0])$ satisfying  $v^{\ep}_0(h^{\ep}_0)=0<v^{\ep}_0(x)$ in $[0,h^{\ep}_0)$ and $v^{\ep}_0\to v_0$ in $C([0,h^{\ep}_0])$ as $\ep\to0$. Let $(u^{\ep},v^{\ep},h^{\ep})$ be the unique solution of \eqref{1.1} with $(u^{\ep}_0,v^{\ep}_0,h^{\ep}_0,\mu^{\ep})$ in place of $(u_0,v_0,h_0,\mu)$.

We claim $\bar{h}{\rrr >}h^{\ep}$ in $[0,T]$. Clearly, it holds true when $t$ is small. If it is not true, there is the smallest $t_0\in(0,T]$ such that $\bar{h}(t_0)=h^{\ep}(t_0)$ and $\bar{h}>h^{\ep}$ in $[0,t_0)$. Obviously, $\bar{h}'(t_0)\le (h^{\ep})'(t_0)$. We now compare $(\bar u,\bar{v})$ and $(u^{\ep},v^{\ep})$ in $[0,t_0]\times\overline{\mathbb{R}}^+$. Define $w=\bar{u}-u^{\ep}$ and $z=\bar{v}-v^{\ep}$. Then
 \bess\left\{\begin{aligned}
&w_t\ge d_1\int_{0}^{\yy}J_1(x-y)w(t,y)\dy-d_1j_1(x)w+c_{11}w+c_{12}z, & &t\in(0,t_0],~x\ge0,\\
&z_t\ge d_2\int_{0}^{h^{\ep}(t)}\!\!J_2(x-y) z(t,y)\dy-d_2j_2(x) z+c_{21}w+c_{22}z, & &t\in(0,t_0],~0\le x\le h^{\ep}(t),\\
&z(t,x)\ge0, ~ 0<t\le t_0, ~ x\ge h^{\ep}(t); ~ ~ ~ w(0,x)>0, ~ x\ge0; ~ ~ ~ z(0,x)>0,& & 0\le x\le h^{\ep}_0,
 \end{aligned}\right.
 \eess
where $c_{ij}\in L^{\yy}([0,t_0]\times\overline{\mathbb{R}}^+)$ for $i,j=1,2$ and $c_{12},c_{21}\ge0$ in $[0,t_0]\times\overline{\mathbb{R}}^+$. Define $\tilde{w}=w+\delta e^{kt}$ and $\tilde{z}=z+\delta e^{kt}$ with small $\delta>0$ and $k>2\big(d_1+d_2+\sum_{i,j=1}^{2}\|c_{ij}\|_{\yy}+1\big)$. Then we have
\bess\left\{\begin{aligned}
&\tilde w_t\ge d_1\int_{0}^{\yy}\!\!J_1(x-y)\tilde w(t,y)\dy-d_1j_1(x)\tilde w+c_{11}\tilde w+c_{12}\tilde z+A_1, \!\!&&t\in(0,t_0],~x\ge0,\\
& \tilde z_t\ge d_2\int_{0}^{h^{\ep}(t)}\!\!\!J_2(x-y) \tilde z(t,y)\dy-d_2j_2(x) \tilde z+c_{21}\tilde w+c_{22}\tilde z+A_2, \!\!& &t\in(0,t_0],~0\le x\le h^{\ep}(t),\\
&\tilde z(t,x)>0,~ 0<t\le t_0, ~ x\ge h^{\ep}(t); ~ ~ ~\tilde w(0,x)\ge\delta,~ x\ge0; ~ ~ ~ \tilde z(0,x)\ge\delta, & &x\in[0,h^{\ep}_0],
 \end{aligned}\right.
 \eess
where $A_i(t,x)=\delta e^{kt}(k-c_{i1}-c_{i2})$, $i=1,2$.
We now prove that $\tilde{w}, \tilde{z}\ge0$ in $[0,t_0]\times\overline{\mathbb{R}}^+$. Argue indirectly, and suppose that there exists $(t,x)\in[0,t_0]\times\overline{\mathbb{R}}^+$ such that $\tilde{w}(t,x)<0$ or $\tilde{z}(t,x)<0$. Since $\tilde{w}(0,x)\ge \delta$ and $\tilde{w}_t$ has a uniform lower bound, there exists the smallest $\tau\in(0,t_0]$ such that $\inf_{x\in\overline{\mathbb{R}}^+}\tilde{w}(\tau,x)=0$, and $\tilde{w}>0$,  $\tilde{z}\ge0$ in  $[0,\tau)\times\overline{\mathbb{R}}^+$; or $\tilde{z}(\tau,x_0)=0$ for some $x_0\in[0,h^{\ep}(\tau))$, and $\tilde{z}>0$, $\tilde{w}\ge0$ in $[0,\tau)\times\overline{\mathbb{R}}^+$. For the first case, we define $\hat{w}=\tilde{w}-\delta$, and then
\bess\left\{\begin{aligned}
&\hat w_t\ge d_1\int_{0}^{\yy}J_1(x-y)\hat w(t,y)\dy-d_1j_1(x)\hat w+c_{11}\hat w+A_1+\delta c_{11}, & &t\in(0,\tau],~x\in\overline{\mathbb{R}}^+,\\
&\hat w(0,x)\ge0, & & x\in\overline{\mathbb{R}}^+.
 \end{aligned}\right.
 \eess
As $A_1+\delta c_{11}=\delta e^{kt}(k-c_{11}-c_{12}+c_{11}e^{-kt})\ge0$, the maximum principle gives $\hat{w}\ge0$, and so $\tilde{w}\ge \delta$ in $[0,\tau]\times\overline{\mathbb{R}}^+$. This contradiction indicates that the first case cannot happen. Similarly, we can show that the second case also cannot happen. Hence $\tilde{w}, \tilde{z}\ge0$ in $[0,t_0]\times\overline{\mathbb{R}}^+$.

By the arbitrariness of $\delta$ and maximum principle,  $w>0$ in $[0,t_0]\times\overline{\mathbb{R}}^+$ and $z>0$ in $[0,t_0]\times[0,h^{\ep}(t))$. Thus we have
 \bess
 0&\ge& \bar{h}'(t_0)-h'^{\ep}(t_0)\ge\mu\int_{0}^{\bar h(t_0)}\int_{\bar h(t_0)}^{\infty}
J_2(x-y)\bar v(t,x)\dy\dx-\mu^{\ep}\int_{0}^{\bar h(t_0)}\int_{\bar h(t_0)}^{\infty}
J_2(x-y)v^{\ep}(t,x)\dy\dx\\
&>&\mu^{\ep}\int_{0}^{\bar h(t_0)}\int_{\bar h(t_0)}^{\infty}
J_2(x-y)\left(\bar v(t,x)-v^{\ep}(t,x)\right)\dy\dx>0.
 \eess
This contradiction implies $\bar{h}>h^{\ep}$ in $[0,T]$. It then follows from the above analysis that $\bar{u}\ge u^{\ep}$, $\bar{v}\ge v^{\ep}$ in $[0,T]\times\overline{\mathbb{R}}^+$.
By the continuous dependence of solution of \eqref{1.1} on $\ep$, the proof is ended.
\end{proof}
From the above comparison principle, we can see that the solution $(u,v,h)$ of \eqref{1.1} is monotonically increasing in $\mu>0$. If $r_2<d_2/2$, from \cite[Lemma 2.3]{LW21} there is a unique $\ell^*>0$ such that $\lambda_p(\mathcal{L}_{(0,\ell^*)}+r_2)=0$ and $\lambda_p(\mathcal{L}_{(0,l)}+r_2)(l-\ell^*)>0$ for any $l\neq\ell^*$.
\begin{theorem}\label{t3.1} Then we have the following conditions governing spreading and vanishing of \eqref{1.1}.

 \sk{\rm(1)}\, If $r_2\ge d_2/2$, then spreading happens;

  \sk{\rm(2)}\, If $r_2<d_2/2$, then spreading happens if $h_0\ge \ell^*$. Moreover, if $h_{\yy}<\yy$, then $h_{\yy}\le \ell^*$;

   \sk{\rm(3)}\,If $r_2<d_2/2$ and $h_0<\ell^*$, then there exists a unique $\mu^*>0$ such that spreading happens if $\mu>\mu^*$, and vanishing occurs if $0<\mu\le\mu^*$.
\end{theorem}
\begin{proof}Conclusions (1) and (2) are easily proved by \cite[Lemma 2.3]{LW21} and Theorem \ref{t2.2},  and we omit the details. Next we focus on the proof of conclusion (3). We first claim that there is a $\underline{\mu}>0$ such that vanishing happens if $\mu<\overline{\mu}$. Let $(\overline{v},\overline{h})$ be the unique solution of \eqref{1.4} with $(d_2,J_2,r_2u(1-u),v_0(x))$ in place of $(d,J,f(u),u_0(x))$.
 By a comparison consideration, we have $h(t)\le \overline{h}(t)$ for $t\ge0$. By virtue of \cite[Theorem 4.3]{LW21}, there is a $\underline{\mu}>0$ such that $\lim_{t\to\yy}\overline{h}(t)<\yy$ if $\mu\le\underline{\mu}$. Hence our claim is proved.

 Then we assert that there is a $\overline{\mu}>0$ such that spreading happens if $\mu>\overline{\mu}$. Let $(\underline{v},\underline{h})$ be the unique solution to \eqref{1.4} with $(d,J,f(u),u_0(x))$ replaced by $(d_2,J_2,r_2u(1-2u),v_0(x))$.
 By a comparison argument again, we see $h(t)\ge\underline{h}(t)$ for $t\ge0$. Moreover, it follows from \cite[Theorem 4.3]{LW21} that there is a $\overline{\mu}>0$ such that $\lim_{t\to\yy}\underline{h}(t)=\yy$ if $\mu>\overline{\mu}$. Thus our assertion follows.

 Based on the above analysis, we can define $\mu^*=\inf\{\Lambda>0: {\rm spreading ~ happens ~ if ~ } \mu>\Lambda\}$.
 Clearly, $\mu^*\in(0,\yy)$. Moreover, spreading occurs if $\mu>\mu^*$, and vanishing happens if $\mu<\mu^*$. To our purpose, it remains to show that vanishing happens if $\mu=\mu^*$. To stress the dependence on $\mu$, we denote the solution of \eqref{1.1} by $(u_{\mu},v_{\mu},h_{\mu})$. Arguing indirectly, we assume that $\lim_{t\to\yy}h_{\mu^*}(t)=\yy$. Then there exists a $T>0$ such that $h_{\mu^*}(T)>\ell^*$. By continuity, there is a small $\ep>0$ such that $h_{\mu}(T)>\ell^*$ when $\mu^*-\ep<\mu<\mu^*$, which contradicts the definition of $\mu^*$. So the proof is end.
\end{proof}

\section{Double free boundaries model}

In this section, we are going to check the dynamics of \eqref{1.1} but with double free boundaries, and to prove that there are analogous conclusions holding true for this problem.
\bes\label{4.1}\left\{\begin{aligned}
&u_t=d_1\int_{-\yy}^{\yy}J_1(x-y)u(t,y)\dy-d_1u+f_1(u,v), & &t>0,~x\in\mathbb{R},\\
&v_t= d_2\int_{g(t)}^{h(t)}J_2(x-y)v(t,y)\dy-d_2v+f_2(u,v), & &t>0,~x\in(g(t),h(t)),\\
&v(t,x)=0,& &t>0, ~ x\notin(g(t),h(t)),\\
&g'(t)=-\mu\int_{g(t)}^{ h(t)}\int_{-\yy}^{g(t)}
J_2(x-y) v(t,x)\dy\dx,& &t>0,\\
& h'(t)=\mu\int_{g(t)}^{ h(t)}\int_{ h(t)}^{\infty}
J_2(x-y)v(t,x)\dy\dx,& &t>0,\\
&u(0,x)= u_0(x)>0, ~ x\in\mathbb{R}; ~ ~ -g(0)=h(0)=h_0; ~ ~ v(0,x)= v_0(x),& &x\in[-h_0,h_0],
 \end{aligned}\right.
 \ees
 where $u_0(x)\in C(\mathbb{R})\cup L^{\yy}(\mathbb{R})$, $v_0(x)\in C([-h_0,h_0])$ and $v_0(\pm h_0)=0<v_0(x)$ in $(-h_0,h_0)$.
Below is firstly the well-posedness result of it.
\begin{theorem}\label{t4.1}Problem \eqref{4.1} has a unique solution $(u,v,g,h)$ defined for $t\ge0$. Moreover, $u,u_t\in C(\overline{\mathbb{R}}^+\times\mathbb{R})$, $v,v_t\in C(\overline{\mathbb{R}}^+\times [g(t),h(t)])$ and $g(t),h(t)\in C^1([0,\yy))$. The following estimates hold
  \[0\le u\le K_1:=\max\{\|u_0\|_{L^{\yy}(\mathbb{R})}, \ a\}, \ \ \ 0\le v\le K_2:=\max\{\|v_0\|_{C([-h_0,h_0])}, \ 1\}.\]
 \end{theorem}

 \subsection{Longtime behaviors of solution}

It is easy to see that $-g(t)$ and $h(t)$ are increasing in $t\ge0$. So setting  $g_\yy=\lim_{t\to\yy}g(t)$ and $h_{\yy}=\lim_{t\to\yy}h(t)$, we can define $h_{\yy}-g_\yy<\yy$ as vanishing case, and $h_{\yy}-g_\yy=\yy$ as spreading case.
For any $a_0>0$, it is well-known that problem
 \bess
(\mathcal{\tilde L}_{(l_1,\,l_2)}+a_0) \phi:=d_2\int_{l_1}^{l_2}J_2(x-y)\phi(y)\dy-d_2\phi+a_0\phi=\lambda\phi ~ ~ \mbox{in}\;\;[l_1,l_2]
\eess
has a unique principal eigenvalue with a positive eigenfunction. Please see \cite[Proposition 3.4]{CDLL} for more details about this eigenvalue problem. Denote $(u,v,g,h)$ a solution of \eqref{4.1} in this section.

\subsubsection{Vanishing case: $h_{\yy}-g_\yy<\yy$}
 \begin{theorem}\label{t4.2} If $h_{\yy}-g_\yy<\yy$, then $\lambda_p(\mathcal{\tilde L}_{(g_\yy,\,h_{\yy})}+r_2)\le 0$, $\lim_{t\to\yy}\|v(t,\cdot)\|_{C([g(t),h(t)])}=0$ and
 \bess\left\{\begin{aligned}
&\lim_{t\to\yy}\max_{x\in[-ct,\,ct]}|u(t,x)-a/2|=0 ~ {\rm for ~ any ~ } c\in(0,c_*), ~ {\rm  ~ if ~ J_1 ~ satisfies ~ {\bf(J2)}},\\
&\lim_{t\to\yy}\max_{x\in[-ct,\,ct]}|u(t,x)-a/2|=0 ~ {\rm for ~ any ~ } c>0, ~ {\rm  ~ if ~ J_1 ~ violates ~ {\bf(J2)}}.
\end{aligned}\right.\eess
 \end{theorem}
 \begin{proof}
 Since it can be proved by arguing as in the proof of Theorem \ref{t2.2}, we omit the details.
 \end{proof}
 \begin{remark}We believe that in vanishing case, the followings hold true for problem \eqref{4.1}.
 \bess\left\{\begin{aligned}
&\lim_{t\to\yy}\max_{x\in[-s(t),\,s(t)]}|u(t,x)-a/2|=0 ~ {\rm for ~ any } ~ 0\le s(t)=t^{\frac{1}{\gamma-1}}o(1) ~ ~ {\rm  ~ if ~}J_1\approx|x|^{-\gamma} ~ {\rm for ~ } \gamma\in(1,2),\\
&\lim_{t\to\yy}\max_{x\in[-s(t),\,s(t)]}|u(t,x)-a/2|=0 ~ {\rm for ~ any } ~ 0\le s(t)=(t\ln t)o(1) ~ ~ {\rm  ~ if ~}J_1\approx|x|^{-2}.
\end{aligned}\right.\eess
In fact, by taking advantage of the arguments in the proof of \cite[Lemmas 6.4 and 6.6]{DN21} and \cite[Theorem 4.1]{LW211} one can prove these results.
 \end{remark}

 \subsubsection{Spreading case: $h_{\yy}-g_\yy=\yy$}

\begin{lemma}\label{l4.1} Assume that $-s_1(t), s_2(t)$ are continuous in $[0,\yy)$ and strictly increasing to $\yy$, $-s_1(0)=s_2(0)=s_0>0$ and $P$ satisfies {\bf(J)}. Let $\underline{s}=\min\{\liminf_{t\to\yy}\frac{-s_1(t)}t,\,\liminf_{t\to\yy}\frac{s_2(t)}t \}$, and constants $d,\alpha,\beta$ be positive. Let $w$ be the unique solution of
\bes\label{4.2}
\left\{\begin{aligned}
&w_t=d\dd\int_{s_1(t)}^{s_2(t)}P(x-y)w(t,y)\dy-d w+w(\alpha-\beta w), && t>0,~s_1(t)<x<s_2(t),\\
&w(t,s_i(t))=0, && t>0,~i=1,2,\\
&w(0,x)=w_0(x), &&-s_0\le x\le s_0,
\end{aligned}\right.
 \ees
 where $w_0(x)\in C([-s_0,s_0])$, $w_0(x)>0=w_0(\pm s_0)$ in $(-s_0,s_0)$. Then the followings hold true:

  \sk{\rm(1)}\, $\lim_{t\to\yy}w(t,x)=\alpha/\beta$ locally uniformly in $\mathbb{R}$.

\sk{\rm(2)}\, Suppose that $P$ satisfies {\bf(J1)} and $\underline{s}\in(0,\yy]$.  Then
\bess\left\{\begin{aligned}
&\lim_{t\to\yy}\max_{x\in[-ct,\,ct]}|w(t,x)-{\alpha}/{\beta}|=0 ~ {\rm for ~ any ~ } c\in(0,\min\{\underline{s},C_*\}), {\rm  ~ if ~} P ~ {\rm  satisfies ~ {\bf(J2)} },\\
&\lim_{t\to\yy}\max_{x\in[-ct,\,ct]}|w(t,x)-{\alpha}/{\beta}|=0 ~ {\rm for ~ any ~ } c\in(0,\underline{s}), ~ {\rm  ~ if ~} P ~ {\rm  violates ~ {\bf(J2)} },\end{aligned}\right.
\eess
where $C_*$ is defined as in Lemma \ref{l2.1}.

  \sk{\rm(3)}\, Suppose that $P$ does not satisfies {\bf(J1)} and $\underline{s}\in(0,\yy]$. Then
    \[\lim_{t\to\yy}\max_{x\in[-ct,\,ct]}|w(t,x)-{\alpha}/{\beta}|=0 ~ {\rm for ~ any ~ } c\in(0,\underline{s}).\]

 \sk{\rm(4)}\, Suppose that there exist $C_1,C_2>0$ such that $P(x)\ge C_1|x|^{-\gamma}$ for $|x|\gg1$ and $\gamma\in(1,2)$, and $\min\{-s_1(t),s_2(t)\}\ge C_2t^{\lambda}$ for $t\gg1$ and some $\lambda\in(1,1/(\gamma-1)]$. Then
 \[\lim_{t\to\yy}\max_{x\in[-t^{\theta_1},\,t^{\theta_2}]}|w(t,x)-{\alpha}/{\beta}|=0 ~ {\rm for ~ any ~ } \theta_1,\theta_2\in(1,\lambda).\]

  \sk{\rm(5)}\, Suppose that there exist $C_1,C_2>0$ such that $P(x)\ge C_1|x|^{-2}$ for $|x|\gg1$, and $\min\{-s_1(t),\\ \nonumber s_2(t)\}\ge C_2t(\ln t)^{\eta}$ for $t\gg1$ and some $0<\eta\le1$.  Then
 \[\lim_{t\to\yy}\max_{x\in[-t[\ln (t+1)]^{\omega_1},\,t[\ln (t+1)]^{\omega_2}]}|w(t,x)-{\alpha}/{\beta}|=0 ~ {\rm for ~ any ~ } \omega_1,\omega_2\in(0,\eta).\]
 \end{lemma}
 \begin{proof}(1) Since it is easy to prove this conclusion, the details are omitted here.

 (2) Clearly, by a simple comparison argument we easily see $\limsup_{t\to\yy}w(t,x)\le \alpha/\beta$ uniformly in $\mathbb{R}$. Thus it remains to study the lower limitation of $w$ as $t\to\yy$. Assume that $P$ satisfies {\bf(J2)}. For any $c<c_1\in (0,\min\{\underline{s},C_*\})$, there exist a large $\zeta>0$ and a small $\delta>0$ such that $c_1<\min\{\underline{s},c^{\zeta}_0-\delta\}$ with $c^{\zeta}_0$ defined as in Lemma \ref{l2.1}. For simplicity, let $F(w)=w(\alpha-\beta w)$. Then by letting $\ep>0$ sufficiently small, we have $F'(w)<0$ for $w\in[\alpha/\beta(1-\ep/2)(1-\ep),\alpha/\beta]$,
 \bes \label{4.3}1-\frac{\ep}{4}>(1-\frac{\ep}{2})(1-\ep) ~ {\rm and} ~ 2(1-\ep)F(\frac{\alpha}{\beta}(1-\frac{\ep}{4}))<F(\frac{\alpha}{\beta}(1-\ep)).\ees
 Moreover, we may choose $M>0$ large enough such that $\phi_{c^{\zeta}_0}(-M)>\alpha/\beta(1-\ep/4)$. Define $\ep_0=\min_{x\in[-M,0]}\{|\phi'_{c^{\zeta}_0}(x)|\}>0$ and $M_0=\max_{w\in[-\alpha/\beta,\alpha/\beta]}|F'(w)|$ and $\hat{\ep}=\delta\ep_0(1-\ep)/2M_0$. Also we can find a large $L>M$ such that $\phi_{c^{\zeta}_0}(-2L+M)\ge\alpha/\beta-\hat{\ep}$.

  Let $\underline{s}(t)=c_1t+L$ for $t\ge0$. By conclusion (1) and $c_1<\underline{s}$, there is $T>0$ such that $w(t,x)\ge(1-\ep)\alpha/\beta$ and $\min\{-s_1(t),s_2(t)\}>c_1t+L$ for $t\ge T$ and $-L\le x\le L$. Define $\tilde{w}(t,x)=w(t+T,x)$. Clearly, $\tilde{w}$ satisfies
 \bess
\left\{\begin{aligned}
&\tilde w_t\ge d\dd\int_{-\underline{s}(t)}^{\underline{s}(t)}P(x-y)\tilde w(t,y)\dy-d\tilde w+F(\tilde w), && t>0,~-\underline{s}(t) <x<\underline{s}(t),\\
&\tilde w(t,\pm \underline{s}(t))>0, && t>0,\\
&\tilde w(0,x)\ge(1-\ep){\alpha}/{\beta}, &&-L\le x\le L.
\end{aligned}\right.
 \eess
 Let \[\underline{w}(t,x)=(1-\ep)[\phi_{c^{\zeta}_0}(x-\underline{s}(t))
 +\phi_{c^{\zeta}_0}(-x-\underline{s}(t))-\alpha/\beta],\;\;\; t\ge0,\;\; x\in[-\underline{s}(t),\underline{s}(t)]. \]
We are going to show that the following inequalities hold true:
  \bes\label{4.4}
\left\{\begin{aligned}
&\underline w_t\le d\dd\int_{-\underline{s}(t)}^{\underline{s}(t)}P(x-y)\underline w(t,y)\dy-d\underline w+F(\underline w), && t>0,~-\underline{s}(t)< x<\underline{s}(t),\\
&\underline w(t,\pm\underline{s}(t))\le0, && t>0,\\
&\underline w(0,x)\le \tilde{w}(0,x), &&-L\le x\le L.
\end{aligned}\right.
 \ees
If \eqref{4.4} is proved, then $w(t+T,x)\ge \underline{w}(t,x)$ for $t>0$ and $-\underline{s}(t)< x<\underline{s}(t)$ by the comparison principle. Moreover,
 \[\max_{x\in[-ct,ct]}|\underline{w}(t,x)-(1-\ep){\alpha}/{\beta}|\le2(1-\ep)
 \big[{\alpha}/{\beta}-\phi_{c^{\zeta}_0}(ct-c_1t-L)\big]\to0 ~ {\rm as } ~t\to\yy,\]
 which, together with the arbitrariness of $\ep$, yields our desired result.

Since the second and third inequalities of \eqref{4.4} are obvious, we next focus on the first one. Direct calculations show
\bess
\underline{w}_t&=&-(1-\ep)c_1[\phi'_{c^{\zeta}_0}(x-\underline{s}(t))
+\phi'_{c^{\zeta}_0}(-x-\underline{s}(t))]\\
&\le&-(1-\ep)(c^{\zeta}_0-\delta)[\phi'_{c^{\zeta}_0}(x-\underline{s}(t))
+\phi'_{c^{\zeta}_0}(-x-\underline{s}(t))]\\
&=&(1-\ep)\delta[\phi'_{c^{\zeta}_0}(x-\underline{s}(t))
+\phi'_{c^{\zeta}_0}(-x-\underline{s}(t))]+(1-\ep)\bigg[d\int_{-\yy}^{\underline{s}(t)}\!\!
P(x-y)\phi_{c^{\zeta}_0}(y-\underline{s}(t))\dy\nonumber\\
&&-d\phi_{c^{\zeta}_0}(x-\underline{s}(t))+d\int_{-\underline{s}(t)}^{\yy}\!\!
P(x-y)\phi_{c^{\zeta}_0}(-y-\underline{s}(t))\dy-d\phi_{c^{\zeta}_0}(-x-\underline{s}(t))\bigg]\\
&&+(1-\ep)\big[F(\phi_{c^{\zeta}_0}(x-\underline{s}(t)))
+F(\phi_{c^{\zeta}_0}(-x-\underline{s}(t)))\big]\\
&=&(1-\ep)\delta[\phi'_{c^{\zeta}_0}(x-\underline{s}(t))+\phi'_{c^{\zeta}_0}(-x-\underline{s}(t))]+d\dd\int_{-\underline{s}(t)}^{\underline{s}(t)}P(x-y)\underline w(t,y)\dy-d\underline w\\
&&+(1-\ep)d\kk(\int_{-\yy}^{-\underline{s}(t)}\!\!P(x-y)[\phi_{c^{\zeta}_0}(y-\underline{s}(t))
-\frac{\alpha}{\beta}]\dy+
\int_{\underline{s}(t)}^{\yy}\!\!P(x-y)[\phi_{c^{\zeta}_0}(-y-\underline{s}(t))
-\frac{\alpha}{\beta}]\dy\rr)\\
&&+(1-\ep)\big[F(\phi_{c^{\zeta}_0}(x-\underline{s}(t)))
+F(\phi_{c^{\zeta}_0}(-x-\underline{s}(t)))\big]\\
&\le&d\dd\int_{-\underline{s}(t)}^{\underline{s}(t)}P(x-y)\underline w(t,y)\dy-d\underline w+A(t,x),
\eess
where
\bess A(t,x)=(1-\ep)\big\{\delta[\phi'_{c^{\zeta}_0}(x-\underline{s}(t))
+\phi'_{c^{\zeta}_0}(-x-\underline{s}(t))]
+F(\phi_{c^{\zeta}_0}(x-\underline{s}(t)))
+F(\phi_{c^{\zeta}_0}(-x-\underline{s}(t)))\big\}.
\eess
We next verify $A\le F(\underline{w})$ for $t\ge0$ and $x\in[-\underline{s}(t),\underline{s}(t)]$. Firstly, we deal with the case $x\in[\underline{s}(t)-M,\underline{s}(t)]$. For such $(t,x)$, from our choice of $L$ we have $0>\phi_{c^{\zeta}_0}(-x-\underline{s}(t))-\alpha/\beta\ge-\hat{\ep}$.
It then follows from the mean value theorem that
\bess
&F(\underline{w})\ge F((1-\ep)\phi_{c^{\zeta}_0}(x-\underline{s}(t)))-M_0(1-\ep)\hat{\ep},\\
&F(\phi_{c^{\zeta}_0}(-x-\underline{s}(t)))=F(\phi_{c^{\zeta}_0}(-x-\underline{s}(t)))-F(\frac{\alpha}{\beta})\le M_0\hat{\ep}.
\eess
Therefore,
\bess
A(t,x)-F(\underline{w})&\le&-(1-\ep)\delta \ep_0+(1-\ep)\left[F(\phi_{c^{\zeta}_0}(x-\underline{s}(t)))+M_0\hat{\ep}\right]\\
&&-F((1-\ep)\phi_{c^{\zeta}_0}(x-\underline{s}(t)))+M_0(1-\ep)\hat{\ep}\\
&\le&-(1-\ep)\delta \ep_0+2M_0(1-\ep)\hat{\ep}<0.
\eess
Moreover we have $A(t,x)\le F(\underline{w})$ for $t\ge0$ and $x\in[-\underline{s}(t),-\underline{s}(t)+M]$ since $A$ and $\underline{w}$ are both even in $x$. It then remains to check the case $x\in[-\underline{s}(t)+M,\underline{s}(t)-M]$. Clearly, in this case we have \[\underline{w}(t,x)\in[\frac{\alpha}{\beta}(1-\ep)(1-\frac{\ep}{2}),(1-\ep)\frac{\alpha}{\beta}].\]
So, by \eqref{4.3}, we have
\[A(t,x)-F(\underline{w})\le(1-\ep)2F(\frac{\alpha}{\beta}(1-\frac{\ep}{4}))-F(\frac{\alpha}{\beta}(1-\ep))<0.\]
Consequently, we prove \eqref{4.4}. The assertion with $P$ satisfying {\bf(J2)} follows.

As for the case $P$ violating {\bf(J2)}, we can argue as above to derive the conclusion. In fact, since $\lim_{\zeta\to\yy}c^{\zeta}_0=\yy$, for any $c\in(0,\underline{s})$ we can choose $\zeta>0$ large sufficiently such that $c<c^{\zeta}_0$. Then a similar lower solution can be constructed to verify this case, and the details are omitted.

(3) We first consider the case $\underline{s}\in(0,\yy)$. Define $P_n$ as in the proof of Lemma \ref{l2.1}. Let $w_n$ be a solution of \eqref{4.1} with $P$ replaced by $P_n$, and $C_n$ be the minimal speed of problem
\bess\left\{\begin{array}{lll}
 \hat d_n\dd\int_{-\yy}^{\yy}\hat P_n(x-y)\phi(y)\dy-\hat d_n\phi+c\phi'+\phi(\hat{d}_n-d+\alpha-\beta \phi)=0, \quad -\yy<x<\yy,\\[1mm]
\phi(-\yy)=(\hat{d}_n-d+\alpha)/{\beta},\ \ \phi(\yy)=0, \ \ \ \phi'(x)\le0,
 \end{array}\right.
 \eess
 where $n$ is large sufficiently such that $\hat{d}_n-d+\alpha>0$.
 Obviously, from a comparison consideration we have $w(t,x)\ge w_n(t,x)$ for $t\ge0$ and $s_1(t)\le x\le s_2(t)$. Since $P$ does not satisfy {\bf(J1)}, it is easy to see (\cite{DLZ}) that $\lim_{n\to\yy}C_n=\yy$. Thus for any $0<c<\underline{s}$, we can choose $N$ large enough such that $c<\min\{\underline{s},C_n\}$. Then by conclusion (2) we have for any large $n$, $\liminf_{t\to\yy}w_n(t,x)\ge (\hat{d}_n-d+\alpha)/\beta$ uniformly in $[-ct,ct]$, and so $\liminf_{t\to\yy}w(t,x)\ge (\hat{d}_n-d+\alpha)/\beta$ uniformly in $[-ct,ct]$. Sending $n\to\yy$, we obtain $\liminf_{t\to\yy}w(t,x)\ge\alpha/\beta$ uniformly in $[-ct,ct]$. Together with $\limsup_{t\to\yy}w(t,x)\le \alpha/\beta$ uniformly in $\mathbb{R}$, we finish the proof of this case.

 When $\underline{s}=\yy$, for any $0<c<c_1<\yy$ there is $T>0$ such that $\min\{-s_1(t),s_2(t)\}>c_1t+s_0:=\tilde{s}(t)$ for $t\ge T$. Define $\tilde{w}(t,x)=w(t+T,x)$ for $t\ge0$ and $x\in[-\tilde{s}(t),\tilde{s}(t)]$. Then $\tilde{w}(t,x)$ satisfies
  \bess
\left\{\begin{aligned}
&\tilde w_t\ge d\dd\int_{-\tilde{s}(t)}^{\tilde{s}(t)}P(x-y)\tilde w(t,y)\dy-d\tilde w+\tilde w(\alpha-\beta \tilde w), && t>0,~-\tilde{s}(t) <x<\tilde{s}(t),\\
&\tilde w(t,\pm \tilde{s}(t))>0, && t>0,\\
&\tilde w(0,x)>0, &&-s_0\le x\le s_0.
\end{aligned}\right.
 \eess
 Let $\underline{w}$ be a solution of
   \bess
\left\{\begin{aligned}
&\underline w_t=d\dd\int_{-\tilde{s}(t)}^{\tilde{s}(t)}P(x-y)\underline w(t,y)\dy-d\underline w+\underline w(\alpha-\beta \underline w), && t>0,~-\tilde{s}(t) <x<\tilde{s}(t),\\
&\underline w(t,\pm \tilde{s}(t))=0, && t>0,\\
&\underline w(0,x)=\underline{w}_0(x)\le \tilde{w}(0,x), &&-s_0\le x\le s_0.
\end{aligned}\right.
 \eess
 By the conclusion of the case $\underline{s}\in(0,\yy)$ and comparison principle, we have $\liminf_{t\to\yy}w(t,x)\ge \alpha/\beta$ uniformly in $[-ct,ct]$. So we get the conclusion (3).

 (4) For any $\theta_1,\theta_2\in(1,\lambda)$, we choose $\theta\in(\max\{\theta_1,\theta_2\},\lambda)$. For small $\ep>0$, define
 \[\underline{s}(t)=(K_1t+K_2)^{\theta}, ~ ~ \underline{w}(t,x)=K_{\ep}(\underline{s}(t)-|x|)/\underline{s}(t)\]
 with $K_{\ep}=\alpha/\beta-\ep$, $K_1$ and $K_2$ to be determined later.

  We next show that there exist adequately small $\ep, K_1>0$ and large $K_2,T_0>0$ such that
    \bes\label{4.5}
\left\{\begin{aligned}
&\underline w_t\le d\dd\int_{-\underline{s}(t)}^{\underline{s}(t)}\!\!P(x-y)\underline w(t,y)\dy-d\underline w+ \underline w(\alpha-\beta\underline{w}), && t>0,~x\in(-\underline{s}(t),\underline{s}(t)),\\
&\underline w(t,\pm \underline{s}(t))=0,&& t>0,\\
&\underline w(0,x)\le w(T,x),&& T\ge T_0, ~ x\in[-K_2^{\theta},K_2^{\theta}].
\end{aligned}\right.
\ees
If \eqref{4.5} is obtained, for such $K_1$ and $K_2$, by our assumptions there is $T_1>T_0$ such that $\min\{-s_1(t),\\ \nonumber s_2(t)\}>\underline{s}(t)$ for $t\ge T_1$. Define $\tilde{w}(t,x)=w(t+T_1,x)$ for $t\ge0$ and $x\in [-\underline{s}(t),\underline{s}(t)]$. Then
  \bess
\left\{\begin{aligned}
&\tilde w_t\ge d\dd\int_{-\underline{s}(t)}^{\underline{s}(t)}\!\!P(x-y)\tilde w(t,y)\dy-d\tilde w+\tilde w(\alpha-\beta \tilde w), && t>0,~-\underline{s}(t) <x<\underline{s}(t),\\
&\tilde w(t,\pm \underline{s}(t))>0, && t>0,\\
&\tilde w(0,x)=w(T_1,x)\ge\underline w(0,x), &&-K_2^{\theta}\le x\le K_2^{\theta}.
\end{aligned}\right.
 \eess
 Employing a comparison method, we have $w(t+T_1,x)=\tilde{w}(t,x)\ge\underline{w}(t,x)$ for $t\ge0$ and $x\in [-\underline{s}(t),\underline{s}(t)]$. In addition, it is easy to see that $\lim_{t\to\yy}\underline{w}(t,x)=K_{\ep}$ uniformly in $[-t^{\theta_1},t^{\theta_2}]$. Our assertion then follows from the arbitrariness of $\ep$.

 Now we are in the position to verify \eqref{4.5}.
 We first claim that there is a positive constant $\tilde{C}_1$ depending only on $P$ and $\gamma$ such that
 \[\int_{-\underline{s}(t)}^{\underline{s}(t)}P(x-y)\underline w(t,y)\dy\ge K_{\ep}\tilde{C}_1\underline{s}^{1-\gamma}(t).\]
 For $x\in[\underline{s}(t)/4,\underline{s}(t)]$, we have that if $K_2$ is large enough,
 \bess
 \dd\int_{-\underline{s}(t)}^{\underline{s}(t)}P(x-y)\underline w(t,y)\dy&=&\dd\int_{-\underline{s}(t)-x}^{\underline{s}(t)-x}P(y)\underline w(t,x+y)\dy\\
 &\ge&\dd\int_{-\underline{s}(t)/4}^{-\underline{s}(t)/8}P(y)\underline w(t,x+y)\dy\\
 &\ge&\frac{K_{\ep}C_1}{\underline{s}(t)}\dd\int_{-\underline{s}(t)/4}^{-\underline{s}(t)/8}
 (-y)^{1-\gamma}\dy\ge\tilde{C}_1K_{\ep}\underline{s}^{1-\gamma}(t).
 \eess
 For $x\in[0,\underline{s}(t)/4]$, by following the similar method as above we see
 \bess
 &&\dd\int_{-\underline{s}(t)}^{\underline{s}(t)}P(x-y)\underline w(t,y)\dy\ge K_{\ep}\dd\int_{\underline{s}(t)/8}^{\underline{s}(t)/4}P(y)\left(\frac{\underline{s}(t)-x-y}{\underline{s}(t)}\right)\dy\ge\tilde{C}_1K_{\ep}\underline{s}^{1-\gamma}(t).
 \eess
 Then noting that $P$ and $\underline{w}$ are both symmetric in $x$, our claim holds true.

 Moreover, from \cite[Lemma 6.3]{DN21} and \eqref{2.9} we have that for small $\ep>0$ and large $K_2>0$,
 \bess
 &&d\dd\int_{-\underline{s}(t)}^{\underline{s}(t)}P(x-y)\underline w(t,y)\dy-d\underline w+ \underline w(\alpha-\beta\underline{w})\\
 &&\quad\ge(d-\frac{\alpha\ep}{4})\dd\int_{-\underline{s}(t)}^{\underline{s}(t)}P(x-y)\underline w(t,y)\dy-d\underline w+\frac{\alpha\ep}{2}\underline{w}+\frac{\alpha\ep}{4}\dd\int_{-\underline{s}(t)}^{\underline{s}(t)}P(x-y)\underline w(t,y)\dy\\
 &&\quad \ge \left[(d-\frac{\alpha\ep}{4})(1-\ep^2)-d+\frac{\alpha\ep}{2}\right]\underline{w}+\frac{\alpha\ep}{4}\dd\int_{-\underline{s}(t)}^{\underline{s}(t)}P(x-y)\underline w(t,y)\dy\\
 &&\quad \ge\frac{\alpha\ep}{4}\dd\int_{-\underline{s}(t)}^{\underline{s}(t)}P(x-y)\underline w(t,y)\dy\ge\frac{\alpha\ep\tilde{C}_1K_{\ep}\underline{s}^{1-\gamma}(t)}{4}.
 \eess
 Since $\theta<1/(\gamma-1)$, we have that for $t>0$ and $x\in(-\underline{s}(t),\underline{s}(t))$,
 \[\underline{w}_t=\frac{K_{\ep}|x|\underline{s}'(t)}{\underline{s}^2(t)}\le \frac{K_{\ep}\theta K_1}{K_1t+K_2}\le\frac{K_{\ep}\theta K_1}{(K_1t+K_2)^{\theta(\gamma-1)}}=K_{\ep}\theta K_1\underline{s}^{1-\gamma}(t)\le\frac{\alpha\ep\tilde{C}_1K_{\ep}\underline{s}^{1-\gamma}(t)}{4}\]
 with $K_1<\frac{\alpha\ep\tilde{C}_1}{4\theta}$. So we have proved the first inequality of \eqref{4.5}. The second identity is obvious. For such $K_2$ and $\ep$ as chosen above, from conclusion (1) there is $T_0>0$ such that $w(T,x)\ge{\alpha}/{\beta}-\ep\ge\underline{w}(0,x)$ for $T\ge T_0$ and $x\in[-K_2^{\theta},K_2^{\theta}]$.
 So the proof of assertion (4) is complete.

 (5) For any $\omega_1,\omega_2\in(0,\eta)$, let $\omega\in(\max\{\omega_1,\omega_2\},\eta)$. For small $\ep>0$, define
 \[\underline{s}(t)=K_1(t+K_2)[\ln(t+K_2)]^{\omega}, ~ ~ \underline{w}(t,x)=K_{\ep}\min\{1,(\underline{s}(t)-|x|)/{(t+K_2)^{1/2}}\}\]
 with $K_{\ep}=\alpha/\beta-\ep$, $K_1$ and $K_2$ to be determined later.

 Similarly, we next prove that there are small $\ep,K_1>0$ and large $K_2,T_0>0$ such that
    \bes\label{4.6}
\left\{\begin{array}{ll}
\underline w_t\le d\dd\int_{-\underline{s}(t)}^{\underline{s}(t)}P(x-y)\underline w(t,y)\dy-d\underline w+ \underline w(\alpha-\beta\underline{w}),\\
 \hspace{35mm}t>0,~x\in(-\underline{s}(t),\underline{s}(t))\setminus\left\{\pm\left[\underline{s}(t)-(t+K_2)^{1/2}\right]\right\},\\
\underline w(t,\pm \underline{s}(t))=0,\hspace{8mm} t>0,\\
\underline w(0,x)\le w(T,x),\quad
 T\ge T_0, \quad x\in[-\underline{s}(0),\underline{s}(0)].
\end{array}\right.
\ees
If \eqref{4.6} is proved, for $K_1$ and $K_2$ as chosen above, there is $T_1>T_0$ such that $\min\{-s_1(t),s_2(t)\}\\ \nonumber>\underline{s}(t)$ for $t\ge T_1$. Define $\tilde{w}(t,x)=w(t+T_1,x)$ for $t\ge0$ and $x\in [-\underline{s}(t),\underline{s}(t)]$. Clearly, $\tilde{w}$ satisfies
  \bess
\left\{\begin{aligned}
&\tilde w_t\ge d\dd\int_{-\underline{s}(t)}^{\underline{s}(t)}P(x-y)\tilde w(t,y)\dy-d\tilde w+\tilde w(\alpha-\beta \tilde w), && t>0,~-\underline{s}(t) <x<\underline{s}(t),\\
&\tilde w(t,\pm \underline{s}(t))>0, && t>0,\\
&\tilde w(0,x)=w(T_1,x)\ge\underline w(0,x), &&-\underline{s}(0)\le x\le \underline{s}(0).
\end{aligned}\right.
 \eess
 It follows from a comparison consideration that $\tilde{w}(t,x)\ge\underline{w}(t,x)$ for $t\ge0$ and $x\in [-\underline{s}(t),\underline{s}(t)]$. Moreover, direct calculations show that $\lim_{t\to\yy}\underline{w}(t,x)=K_{\ep}$ uniformly in $[-t[\ln (t+1)]^{\omega_1},\,t[\ln (t+1)]^{\omega_2}]$. By the arbitrariness of $\ep$, we immediately obtain the desired result.

 Now we are ready to prove \eqref{4.6}.
 We first claim that when $K_2$ is large sufficiently, for $x\in[-\underline{s}(t),-\underline{s}(t)+(t+K_2)^{1/2}]\cup[\underline{s}(t)-(t+K_2)^{1/2},\underline{s}(t)]$,
 \[\int_{-\underline{s}(t)}^{\underline{s}(t)}P(x-y)\underline w(t,y)\dy\ge \frac{K_{\ep}C_1\ln(t+K_2)}{2(t+K_2)^{1/2}}.\]
 For $x\in[\underline{s}(t)-3(t+K_2)^{1/2}/4,\underline{s}(t)]$, we have that if $K_2$ is large enough,
 \bess
 \dd\int_{-\underline{s}(t)}^{\underline{s}(t)}P(x-y)\underline w(t,y)\dy&\ge& K_{\ep}\dd\int_{\underline{s}(t)-(t+K_2)^{1/2}}^{\underline{s}(t)}P(x-y)
 \frac{\underline{s}(t)-y}{(t+K_2)^{1/2}}\dy\\
 &\ge&\frac{K_{\ep}C_1}{(t+K_2)^{1/2}}\dd\int_{-(t+K_2)^{1/2}/4}^{-(t+K_2)^{1/4}/4}(-y)^{-1}\dy=\frac{K_{\ep}C_1\ln(t+K_2)}{4(t+K_2)^{1/2}}.
 \eess
For $x\in[\underline{s}(t)-(t+K_2)^{1/2},\underline{s}(t)-3(t+K_2)^{1/2}/4]$, we similarly get
 \bess
 \dd\int_{-\underline{s}(t)}^{\underline{s}(t)}P(x-y)\underline w(t,y)\dy&\ge& \frac{K_{\ep}}{(t+K_2)^{1/2}}\dd\int_{(t+K_2)^{1/4}/4}^{(t+K_2)^{1/2}/4}P(y)
 \left(\underline{s}(t)-x-y\right)\dy\\
 &\ge&\frac{K_{\ep}C_1}{(t+K_2)^{1/2}}\dd\int_{(t+K_2)^{1/4}/4}^{(t+K_2)^{1/2}/4}y^{-1}\dy=\frac{K_{\ep}C_1\ln(t+K_2)}{4(t+K_2)^{1/2}}.
 \eess
 By the symmetry of $P$ and $\underline{w}$ about $x$, our claim is proved.

In view of \cite[Lemma 6.5]{DN21} and \eqref{2.9}, by choosing $\ep$ small and $K_2$ large enough we have for $x\in[-\underline{s}(t),-\underline{s}(t)+(t+K_2)^{1/2}]\cup[\underline{s}(t)-(t+K_2)^{1/2},\underline{s}(t)]$,
 \bess
 &&d\dd\int_{-\underline{s}(t)}^{\underline{s}(t)}P(x-y)\underline w(t,y)\dy-d\underline w+ \underline w(\alpha-\beta\underline{w})\\
 &\ge&(d-\frac{\alpha\ep}{4})\dd\int_{-\underline{s}(t)}^{\underline{s}(t)}P(x-y)\underline w(t,y)\dy-d\underline w+\frac{\alpha\ep}{2}\underline{w}+\frac{\alpha\ep}{4}\dd\int_{-\underline{s}(t)}^{\underline{s}(t)}P(x-y)\underline w(t,y)\dy\\
 &\ge& \left[(d-\frac{\alpha\ep}{4})(1-\ep^2)-d+\frac{\alpha\ep}{2}\right]\underline{w}+\frac{\alpha\ep}{4}\dd\int_{-\underline{s}(t)}^{\underline{s}(t)}P(x-y)\underline w(t,y)\dy\\
 &\ge&\frac{\alpha\ep}{4}\dd\int_{-\underline{s}(t)}^{\underline{s}(t)}P(x-y)\underline w(t,y)\dy\ge\frac{K_{\ep}\alpha \ep C_1\ln(t+K_2)}{16(t+K_2)^{1/2}}.
 \eess
 On the other hand, $x\in(-\underline{s}(t),-\underline{s}(t)+(t+K_2)^{1/2})\cup(\underline{s}(t)-(t+K_2)^{1/2},\underline{s}(t))$,
 \[\underline{w}_t\le\frac{\ln(t+K_2)}{(t+K_2)^{1/2}}\left[K_1K_{\ep}+K_1\omega K_{\ep}\right]\le\frac{K_{\ep}\alpha \ep C_1\ln(t+K_2)}{16(t+K_2)^{1/2}}\]
 provided that $K_1\le\frac{C_1\alpha\ep}{16(1+\omega)}$.
 For $x\in(-\underline{s}(t)+(t+K_2)^{1/2},\underline{s}(t)-(t+K_2)^{1/2})$, we have
 \bess
 &&d\dd\int_{-\underline{s}(t)}^{\underline{s}(t)}P(x-y)\underline w(t,y)\dy-d\underline w+ \underline w(\alpha-\beta\underline{w}) \ge\frac{\alpha\ep}{4}\dd\int_{-\underline{s}(t)}^{\underline{s}(t)}P(x-y)\underline w(t,y)\dy\ge0=\underline{w}_t.
 \eess
 So we have proved the first inequality of \eqref{4.6}. Obviously, the second identity holds true. For such $K_1,K_2$ and $\ep$ as chosen above, by conclusion (1) there is $T_0>0$ such that $w(T,x)\ge{\alpha}/{\beta}-\ep\ge\underline{w}(0,x)$ for $T\ge T_0$ and $x\in[-K_1K_2(\ln K_2)^{\omega},K_1K_2(\ln K_2)^{\omega}]$.
 This completes the proof of conclusion (5). Consequently, we finish the proof of all assertions.
 \end{proof}
With the help of Lemma \ref{l4.1}, it is not hard to prove the following theorem by arguing as in the proof of Theorem \ref{t2.3}. The details are omitted here.
 \begin{theorem} \label{t4.3}If $h_{\yy}-g_\yy=\yy$, then we have the following conclusions:

 \sk{\rm(1)}\, $\lim_{t\to\yy}u(t,x)=u^*$ and $\lim_{t\to\yy}v(t,x)=v^*$ locally uniformly in $\mathbb{R}$.

 \sk{\rm(2)}\, If $J_1,J_2$ satisfy {\bf(J2)} and {\bf(J1)} respectively, for $c_*$ and $c_0$ defined as in Theorem \ref{t2.3} we get
 \[\lim_{t\to\yy}\max_{x\in[-ct,\,ct]}\left\{|u(t,x)-u^*|+|v(t,x)-v^*|\right\}=0 ~ {\rm for ~ any ~ } c\in(0,\min\{c_*,c_0\}).\]

 \sk{\rm(3)}\, If $J_1$ violates {\bf(J2)} and $J_2$ satisfies {\bf(J1)}, then
 \[\lim_{t\to\yy}\max_{x\in[-ct,\,ct]}\left\{|u(t,x)-u^*|+|v(t,x)-v^*|\right\}=0 ~ {\rm for ~ any ~ } c\in(0,c_0).\]

  \sk{\rm(4)}\, If {\bf(J2)} is true for $J_1$ and $J_2$ violates {\bf(J1)}, then
 \[\lim_{t\to\yy}\max_{x\in[-ct,\,ct]}\left\{|u(t,x)-u^*|+|v(t,x)-v^*|\right\}=0 ~ {\rm for ~ any ~ } c\in(0,c_*).\]

  \sk{\rm(5)}\, If $J_1$ violates {\bf(J2)} and $J_2$ does not satisfy {\bf(J1)}, then
 \[\lim_{t\to\yy}\max_{x\in[-ct,\,ct]}\left\{|u(t,x)-u^*|+|v(t,x)-v^*|\right\}=0 ~ {\rm for ~ any ~ } c>0.\]

 \sk{\rm(6)}\, If there is $C>0$ such that $\min\{J_1(x),J_2(x)\}\ge C|x|^{-\gamma}$ for $|x|\gg1$ and $\gamma\in(1,2)$, then
 \[\lim_{t\to\yy}\max_{x\in[-t^{\theta_1},\,t^{\theta_2}]}\left\{|u(t,x)-u^*|+|v(t,x)-v^*|\right\}=0 ~ {\rm for ~ any ~ } 1<\theta_1,\theta_2<1/(\gamma-1).\]

 \sk{\rm(7)}\, If there is $C>0$ such that $\min\{J_1(x),J_2(x)\}\ge C|x|^{-2}$ for $|x|\gg1$, then
 \[\lim_{t\to\yy}\max_{x\in[-t[\ln (t+1)]^{\omega_1},\,t[\ln (t+1)]^{\omega_2}]}\left\{|u(t,x)-u^*|+|v(t,x)-v^*|\right\}=0 ~ {\rm for ~ any ~ }0<\omega_1,\omega_2<1.\]
 \end{theorem}

 \subsection{Criteria for spreading and vanishing}

 Some criteria for spreading and vanishing of \eqref{4.1} will be given in this subsection. To this end, a comparison principle is first shown. We omit the proofs of them since they can be proved by employing the similar methods as in the previous arguments.
 \begin{lemma}[Comparison principle]\label{l4.2} Assume that $\bar{u},\bar{u}_t\in C([0,T]\times\mathbb{R})$, $\bar{u}\in L^{\yy}([0,T]\times\mathbb{R})$, $\bar{v},\bar{v}_t\in C([0,T]\times[\bar{g}(t),\bar{h}(t)])$ and $\bar{g}(t),\bar{h}(t)\in C^1([0,T])$. If $(\bar{u},\bar{v},\bar{g},\bar{h})$ satisfies
\bess\left\{\begin{aligned}
&\bar u_t\ge d_1\int_{-\yy}^{\yy}J_1(x-y)\bar u(t,y)\dy-d_1\bar u+f_1(\bar u,\bar v), & &t\in(0,T],~x\in\mathbb{R},\\
&\bar v_t\ge d_2\int_{\bar{g}(t)}^{\bar h(t)}J_2(x-y)\bar v(t,y)\dy-d_2\bar v+f_2(\bar u,\bar v), & &t\in(0,T],~x\in(\bar{g}(t),\bar h(t)),\\
&\bar v(t,x)\ge0,& &t\in(0,T], ~ x\notin(\bar{g}(t),\bar h(t)),\\
&\bar g'(t)\le-\mu\int_{\bar{g}(t)}^{\bar h(t)}\int_{-\yy}^{\bar{g}(t)}
J_2(x-y)\bar v(t,x)\dy\dx,& &t\in(0,T],\\
&\bar h'(t)\ge\mu\int_{\bar{g}(t)}^{\bar h(t)}\int_{\bar h(t)}^{\infty}
J_2(x-y)\bar v(t,x)\dy\dx,& &t\in(0,T],\\
&\bar u(0,x)\ge u_0(x), ~ x\in\mathbb{R}; ~ ~ \bar h(0)\ge h_0>0; ~ ~ \bar v(0,x)\ge v_0(x),& &x\in[-\bar h(0),\bar h(0)],
 \end{aligned}\right.
 \eess
 then
 \[\bar{h}(t)\ge h(t), ~ ~ \bar{g}(t)\le g(t), ~ ~ t\in[0,T]; ~ ~ ~ \bar{u}(t,x)\ge u(t,x) ~ ~ \bar{v}(t,x)\ge v(t,x), ~ (t,x)\in[0,T]\times\mathbb{R}.\]
\end{lemma}

The above comparison principle implies that solution $(u,v,g,h)$ of \eqref{4.1} is increasing in $\mu$. By \cite[Proposition 3.4]{CDLL}, there is a unique $\tilde \ell>0$ such that $\lambda_p(\mathcal{\tilde L}_{(l_1,l_2)}+r_2)=0$ if $l_2-l_1=\tilde{\ell}$ and $\lambda_p(\mathcal{\tilde L}_{(l_1,l_2)}+r_2)(l_2-l_1-\tilde{\ell})>0$ if $l_2-l_1\neq\tilde{\ell}$.

\begin{theorem}\label{t4.4}Let $(u,v,g,h)$ be a solution of \eqref{4.1}. Then the following conclusions hold true:

\sk{\rm(1)}\, If $r_2\ge d_2$, then spreading happens.

\sk{\rm(2)}\, If $r_2<d_2$ and $2h_0\ge \tilde \ell$, then spreading happens. Moreover, if $h_{\yy}-g_\yy<\yy$, $h_{\yy}-g_\yy\le \tilde \ell$.

   \sk{\rm(3)}\,If $r_2<d_2$ and $2h_0<\tilde \ell$, then there exists a unique $\tilde \mu>0$ such that spreading happens when $\mu>\tilde \mu$, and vanishing occurs when $0<\mu\le\tilde \mu$.
\end{theorem}

\end{document}